\documentclass[reqno, 11pt]{amsart}
\usepackage[T1]{fontenc}
\usepackage[utf8]{inputenc}
\usepackage{geometry}
\usepackage{amssymb}
\usepackage{amsmath}
\usepackage{indentfirst}
\usepackage{graphicx}
\usepackage[colorlinks=true]{hyperref}
\usepackage{xcolor}
\usepackage{microtype}
\AtBeginDocument{\begingroup\normalfont\footnotesize\fontencoding{OT1}\selectfont\endgroup}

\newtheorem{Theorem}{Theorem}[section]
\newtheorem{Corollary}[Theorem]{Corollary}
\newtheorem{Lemma}[Theorem]{Lemma}
\newtheorem{Proposition}[Theorem]{Proposition}
\newtheorem{Remark}[Theorem]{Remark}
\newtheorem{Definition}[Theorem]{Definition}
\newcommand{\C}{\mathbb{C}}
\newcommand{\R}{\mathbb{R}}
\numberwithin{equation}{section}
\usepackage{tikz}
\usepackage{subcaption}
\usetikzlibrary{arrows.meta}

\begin{document}

\title[Bound states of nonlinear Dirac equations on graphs]{Multiplicity of bound states for a class of nonlinear Dirac equations on generalized star graphs with potentials}

\author[G. Gu]{Guangze Gu}
\address{ Guangze Gu: \endgraf Department of Mathematics, Yunnan Normal University, Kunming, China \endgraf Yunnan Key Laboratory of Modern Analytical Mathematics and Applications, Kunming, China }
\email{guangzegu@163.com}

\author[Z. Li]{Ziwei Li}
\address{ Ziwei Li: \endgraf Department of Mathematics, Yunnan Normal University, Kunming, China \endgraf Yunnan Key Laboratory of Modern Analytical Mathematics and Applications, Kunming, China }
\email{ziweili1024@163.com}

\author[M. Ruzhansky]{Michael Ruzhansky}
\address{ Michael Ruzhansky: \endgraf Department of Mathematics: Analysis, Logic and Discrete Mathematics, Ghent University, Belgium \endgraf School of Mathematical Sciences, Queen Mary University of London, United Kingdom }
\email{michael.ruzhansky@ugent.be}

\author[Z. Yang]{Zhipeng Yang}
\address{ Zhipeng Yang: \endgraf Department of Mathematics, Yunnan Normal University, Kunming, China \endgraf Department of Mathematics: Analysis, Logic and Discrete Mathematics, Ghent University, Belgium }
\email{yangzhipeng326@163.com}

\thanks{Corresponding author: Z. Yang.}

\subjclass[2020]{35R02; 35Q41; 81Q35}
\keywords{Nonlinear Dirac equation; Variational methods; Quantum graph}

\date{}

\begin{abstract}
In this paper we investigate the existence and multiplicity of bound states for a class of nonlinear Dirac equations on a generalized star graph \(\mathcal G\):
\[
-ic\,\sigma_1\,\partial_xu+mc^2\sigma_3u+V(x)u=f(|u|)u,\qquad x\in\mathcal G,
\]
where \(V:\mathcal G\to\mathbb R\) is continuous, \(f\in C^1([0,\infty))\), \(m>0\), \(c>0\), and \(\sigma_1,\sigma_3\) are the Pauli matrices. Under suitable monotonicity and growth assumptions on \(f\), a uniform spectral gap and sufficiently long constant-potential intervals on \(k\) terminal edges yield at least \(2^k-1\) bound states, distinct up to a constant phase. The constant potentials on these intervals may differ. The construction uses half-line solutions and a uniform inverse estimate. For potentials tending to a constant at infinity, we prove compactness below the full-line ground level. In the flat-well setting, the solutions localized near individual terminal vertices lie below this level and provide a ground state through the generalized Nehari minimization. In particular, one localized solution is obtained for each minimum component containing a long terminal interval.
\end{abstract}

\maketitle

\section{Introduction and main results}

This paper concerns the existence and multiplicity of solutions to the following class of nonlinear Dirac equations (NLDE) on a generalized star graph:
\begin{equation}\label{eq1.1}
-i\,c\,\sigma_1\,\partial_x u + m c^2 \sigma_3 u + V(x)\,u = f(|u|)\,u, \quad x\in \mathcal{G}.
\end{equation}
Here $u:\mathcal G\to\mathbb C^{2}$, $V:\mathcal{G}\to\mathbb{R}$ and $f:[0,\infty)\to\mathbb{R}$ are continuous functions, \(m>0\) denotes the mass, \(c>0\) the speed of light, and \(\sigma_1,\sigma_3\) are the Pauli matrices
\[
\sigma_1=\begin{pmatrix}0&1\\[2pt]1&0\end{pmatrix}, \qquad \sigma_3=\begin{pmatrix}1&0\\[2pt]0&-1\end{pmatrix}.
\]
Throughout the paper, $\mathcal G$ is a connected noncompact metric graph with finitely many edges and nonempty compact core $\mathcal K$. We fix a reference point $0\in\mathcal K$ and denote by $d(\cdot,\cdot)$ the geodesic distance on $\mathcal G$. A function $u:\mathcal G\to\mathbb C^{2}$ is identified with a family $u=(u_e)_{e\in\mathcal E}$, where $u_e:I_e\to\mathbb C^{2}$ is the restriction of $u$ to the edge $I_e$. The derivative $\partial_x u$ in \eqref{eq1.1} is understood edgewise, that is, $(\partial_x u)_e=(u_e)'$ on each $I_e$. Equation \eqref{eq1.1} is understood together with Kirchhoff-type vertex conditions for the Dirac operator, so that the corresponding realization is self-adjoint; see Section~2.

For the results on minimum wells and ground states, we impose the following assumptions on the potential \(V\):
\begin{itemize}
\item[\((V_1)\)] $V:\mathcal{G}\to\mathbb{R}$ is continuous and there exists $V_\infty\in\mathbb R$ such that for every $\varepsilon>0$ there exists $R>0$ satisfying
\[
|V(x)-V_\infty|<\varepsilon\qquad\text{for all }x\in\mathcal G\setminus B_R(0).
\]
Moreover,
\[
-mc^2<V_0\le V(x)\le V_\infty<mc^2\qquad\text{for all }x\in\mathcal G,
\]
where
\[
V_0=\min_{x\in\mathcal G}V(x).
\]
\item[\((V_2)\)] The minimum set of \(V\) has exactly \(k\) connected components:
\[
Z:=\{x\in\mathcal G:V(x)=V_0\}=Z_1\cup Z_2\cup\cdots\cup Z_k.
\]
Each \(Z_i\) is a compact connected subgraph of \(\mathcal G\) with nonempty interior. Moreover, there exist constants \(\rho_0>0\) and \(\delta>0\) such that the tubular neighborhoods
\[
Z_i^{\rho_0}:=\{x\in\mathcal G:d(x,Z_i)<\rho_0\},\qquad i=1,\dots,k,
\]
are pairwise disjoint and
\[
V(x)\ge V_0+\delta\qquad\text{for all }x\in\mathcal G\setminus\bigcup_{i=1}^kZ_i^{\rho_0}.
\]
\item[\((V_3)\)] \(\mathcal G=\mathcal G_0\cup p_1\cup\cdots\cup p_k\), where \(\mathcal G_0\) is a fixed connected noncompact graph with finitely many edges and each \(p_i\simeq[0,R_i+2]\), with \(R_i\ge2\), is a terminal edge: \(x=0\) is its free vertex \(z_i\), and \(x=R_i+2\) is attached to a fixed vertex of \(\mathcal G_0\). Moreover, \([0,R_i]\subset Z_i\). We choose \(0=z_1\) and set \(R_*=\min_iR_i\).
\end{itemize}

Thus \(V\) attains its minimum on \(k\) separated flat wells, each containing a flat interval of length \(R_i\) that starts at a terminal vertex. The lengths \(R_i\) vary while \(\mathcal G_0\), the attachment vertices and \(V_0,V_\infty\) remain fixed.

Next, we assume that the nonlinearity $f:[0,\infty)\to\mathbb{R}$ is continuous and satisfies:
\begin{itemize}
\item[\((f_1)\)] $f\in C^{1}([0,\infty))$, $f(0)=0$, $f(t)\ge 0$ for all $t\ge 0$, and $f'(t)>0$ for all $t>0$. Moreover, there exist $C>0$ and $p>2$ such that
\[
0\le f(t)\le C\bigl(1+t^{p-2}\bigr)\qquad \forall\, t\ge 0.
\]

\item[\((f_2)\)] There exists \(\theta>2\) such that
\[
0< \theta\,F(t)\le f(t)\,t^2 \quad \text{for all } t>0,
\]
where \(F(t)=\int_0^t f(s)\,s\,ds\).
\end{itemize}

The stationary equation \eqref{eq1.1} can be viewed as the standing wave reduction of the time-dependent NLDE
\[
-i\,\partial_t \psi = i\,c\,\sigma_1\,\partial_x \psi - m c^2 \sigma_3 \psi - V(x)\,\psi + f(|\psi|)\,\psi, \quad (t,x)\in\mathbb{R}\times\mathcal{G},
\]
whose solutions are two-component spinors \(\psi=(\psi^1,\psi^2)^{T}\). Using the standing-wave ansatz \(\psi(t,x)=e^{-i\omega t}u(x)\), one obtains
\[
-ic\,\sigma_1\,\partial_x u + m c^2 \sigma_3 u + (V(x)-\omega)\,u = f(|u|)\,u, \quad x\in\mathcal G.
\]
Absorbing the frequency \(\omega\) into the potential yields \eqref{eq1.1}.

In the simplified setting of the infinite $3$-star graph, namely three half lines meeting at a vertex, Sabirov et al.~\cite{MR3871194} proposed the study of NLDE on networks with the Dirac operator
\[
\mathcal{D}= -i\,c\,\sigma_1\,\partial_x + m c^2 \sigma_3 .
\]
Subsequently, Borrelli et al.~\cite{MR3934110} investigated bound states and the nonrelativistic limit of NLDE on noncompact quantum graphs. To handle more intricate graph topologies, they considered Kirchhoff-type vertex conditions for extensions of the Dirac operator and introduced a localized nonlinearity, leading to the model
\begin{equation}\label{eq1.2}
\mathcal{D}\psi - \chi_{\mathcal K}\,|\psi|^{p-2}\psi = \omega\,\psi,
\end{equation}
where \(\chi_{\mathcal K}\) is the characteristic function of the compact core \(\mathcal K\subset\mathcal G\). They established existence and multiplicity of bound states as critical points of the associated action functional.

For the nonlinearity acting on the whole graph, Borrelli, Carlone and Tentarelli~\cite{MR4200758} proved local well-posedness in the operator domain and constructed standing waves on infinite star graphs by bifurcation. The power nonlinearities in these models depend on the Euclidean modulus of the spinor.

On periodic quantum graphs, Yang and Zhu~\cite{Yang-Zhu} considered Dirac equations with a periodic mass term and a general nonlinear potential. Under asymptotically linear or superquadratic assumptions, they obtained bound states and, for an even nonlinear potential, infinitely many geometrically distinct bound states, using strongly indefinite variational methods and compactness modulo lattice translations.

In Euclidean space \(\mathbb{R}^3\), the stationary nonlinear Dirac equation reads
\begin{equation}\label{eq1.3}
-\,i\,c\,\boldsymbol{\alpha}\!\cdot\!\nabla \psi \;+\; m c^2 \beta\,\psi \;+\; V(x)\,\psi \;=\; f(|\psi|)\,\psi, \qquad x\in\mathbb{R}^3,
\end{equation}
where \(\psi:\mathbb{R}^3\to\mathbb{C}^4\), \(\boldsymbol{\alpha}=(\alpha_1,\alpha_2,\alpha_3)\) and \(\beta\) are the Dirac matrices satisfying
\[
\{\alpha_j,\alpha_k\}=2\delta_{jk}I_4,\qquad \{\alpha_j,\beta\}=0,\qquad \beta^2=I_4,
\]
for instance in the standard representation
\[
\alpha_k=\begin{pmatrix}0&\sigma_k\\ \sigma_k&0\end{pmatrix}, \qquad \beta=\begin{pmatrix}I_2&0\\ 0&-I_2\end{pmatrix},
\]
with \(\sigma_k\) the Pauli matrices. Here \(m>0\) is the mass, \(c>0\) the speed of light, and \(V:\mathbb{R}^3\to\mathbb{R}\) is an external potential. In recent years, the existence of solutions to \eqref{eq1.3} under various assumptions on \(V\) and \(f\) has been extensively investigated; see, e.g., \cite{MR2232435, MR2652161, MR2999413, MR2891353, MR2434900, MR2450893, MR3317806, MR3567492, MR3134574, MR3198134}.

A typical structural condition ensuring the existence of bound states is the potential-well assumption
\[
0<V_0=\min_{x\in\mathbb{R}^3}V(x) \;<\; \liminf_{|x|\to\infty} V(x),
\]
which was already used by Rabinowitz \cite{MR1162728} in the study of nonlinear Schr\"odinger equations in \(\mathbb{R}^N\).

Recently, Alves et al.~\cite{MR4604830}, building on ideas of Cao and Noussair~\cite{MR1409663}, showed that the number of global minimum points of \(V\) is directly related to the number of standing wave solutions of nonlinear Dirac equations in \(\mathbb{R}^3\). Their proof adapts the generalized Nehari manifold approach of Szulkin and Weth~\cite{MR2557725,MR2768820} together with a barycenter scheme to detect multiple critical points.

The prescribed-mass problem associated with \eqref{eq1.2} was studied by He and Ji~\cite{HeJi2026}, who obtained normalized solutions under conditions on the power, the nonlinear coefficient and the graph. In this formulation, the frequency is a Lagrange multiplier associated with the \(L^2\) constraint. Chen and Guo~\cite{ChenGuo2026} studied multiplicity of normalized bound states and their nonrelativistic limit. For the Cauchy problem on \(N\)-star graphs, Xing and Yang established local well-posedness results for power nonlinearities~\cite{XingYang2026Power} and for a mixed-sign quadratic interaction~\cite{XingYang2026Quadratic}, using function spaces defined by the spectral resolution of the Dirac operator.

For autonomous equations on generalized star graphs, Yang~\cite{YangGap2026} established a strict comparison between the graph ground level and the full-line level when the graph contains a sufficiently long pendant path. This comparison is uniform for constant potentials in compact subintervals of the spectral gap.

For scalar nonlinear Schr\"odinger equations on metric graphs, the existence of bound states depends on both the topology and the edge lengths. Adami, Serra and Tilli~\cite{AST2015,AST2016} established topological obstructions and metric threshold conditions for ground states with prescribed mass. More recent results describe several distinct aspects of this dependence. Dovetta~\cite{Dovetta2025} proved nonuniqueness of normalized ground states for powers close to the critical exponent on graphs admitting ground states at every mass. Dovetta, Jeanjean and Serra~\cite{DJS2025} obtained positive normalized solutions with small prescribed mass in the supercritical regime on periodic graphs and on finite-edge noncompact graphs satisfying suitable topological conditions. Agostinho, Correia and Tavares~\cite{ACT2025} studied action ground states and positive bound states on graphs whose half lines, loops and terminal edges meet at a single vertex, including classification and symmetry-breaking results.

Terminal edges also permit a direct construction of localized solutions. Dovetta, Ghimenti, Micheletti and Pistoia~\cite{DGMP2020} described positive solutions of low action and constructed solutions with one or several peaks on compact graphs at large frequencies, using half-line profiles and Lyapunov--Schmidt reduction. Berkolaiko, Marzuola and Pelinovsky~\cite{BMP2021} constructed stationary states localized on a single edge of bounded or unbounded graphs in the large-mass regime. Their method matches the state on that edge with an exponentially smaller solution on the remaining graph through Dirichlet-to-Neumann manifolds.

Our construction uses half-line solutions on terminal intervals, in the spirit of the scalar localization arguments in \cite{DGMP2020,BMP2021}. For the strongly indefinite Dirac problem, the real invariant subspace \((a,ib)^T\) excludes the phase direction, while the Wronskian identity and the terminal condition exclude the translation direction from the linearized kernel. Together with the spectral gap, Lemmas~\ref{Lem3.28} and \ref{Lem3.29} give the uniform inverse estimate used in Section~\ref{Subsec4.1}.

The local construction requires only a uniform bound on the potential and constancy on the terminal intervals. We state it first in this form. Two solutions are regarded as equivalent when they differ by a constant phase \(u\mapsto e^{i\vartheta}u\).

\begin{Theorem}\label{Thm1.1}
Fix \(m,c>0\), a nonlinearity \(f\) satisfying \((f_1)\)--\((f_2)\), a connected noncompact metric graph \(\mathcal G_0\) with finitely many edges, an integer \(k\ge1\), and \(k\) attachment vertices of \(\mathcal G_0\), not necessarily distinct. Fix also \(0\le\bar v<mc^2\) and \(\lambda_1,\ldots,\lambda_k\in[-\bar v,\bar v]\). There exists \(R_0>0\), depending only on these fixed data, with the following property.

For every \(R_1,\ldots,R_k\ge2\) with \(R_*:=\min_iR_i\ge R_0\), attach \(p_i=[0,R_i+2]\) to \(\mathcal G_0\) at \(x=R_i+2\), leaving \(x=0\) as a terminal vertex \(z_i\) of degree one, and denote the resulting graph by \(\mathcal G\). For every \(V\in C(\mathcal G,\mathbb R)\) satisfying
\[
\|V\|_\infty\le\bar v,\qquad V|_{p_i}=\lambda_i\quad\text{on }[0,R_i],\quad i=1,\ldots,k,
\]
problem \eqref{eq1.1} has a nontrivial solution \(u_S\) for every nonempty subset \(S\subset\{1,\ldots,k\}\), localized on the terminal intervals indexed by \(S\). These solutions belong to pairwise distinct constant-phase orbits. In particular, there are at least \(2^k-1\) such orbits. The quantitative approximation and energy estimates are given in Lemma~\ref{Lem4.2}.
\end{Theorem}

In the setting \((V_1)\)--\((V_3)\), take \(\lambda_i=V_0\). The singleton solutions \(u^{(i)}=u_{\{i\}}\) are localized in the corresponding minimum components \(Z_i\). Their energy lies below the full-line ground level at infinity. This estimate also supplies the upper bound needed for ground-state minimization.

We use the generalized Nehari level
\[
\ell=\inf_{u\in\mathcal M}I(u),\qquad \mathcal M=\{u\in Y\setminus Y^-:I'(u)[u]=0,\ I'(u)[v]=0\text{ for all }v\in Y^-\},
\]
where \(Y=Y^+\oplus Y^-\) is the energy space and \(I\) is the associated functional; see Section~\ref{Sec2}. The autonomous graph level \(d_\lambda\) and the full-line level \(d_\lambda^\infty\) are defined in Section~\ref{Subsec3.1}.

\begin{Theorem}\label{Thm1.2}
Fix \(m,c,f,\mathcal G_0\), the attachment vertices and \(k\) as in Theorem~\ref{Thm1.1}, and fix \(-mc^2<V_0<V_\infty<mc^2\). There exists \(R_0>0\), depending only on these data and \(V_0,V_\infty\), such that the following holds for every graph and continuous potential satisfying \((V_1)\)--\((V_3)\) with \(R_*\ge R_0\).

There are \(k\) solutions \(u^{(i)}\), pairwise inequivalent under constant phase, with exponentially small relative \(L^2\) mass outside \(Z_i\). Moreover, \(\ell\) is attained by a nontrivial critical point \(u_*\) of \(I\), and
\[
0<d_{V_0}\le I(u_*)=\ell\le d_{V_\infty}<d_{V_\infty}^{\infty},\qquad \ell\le\min_i I(u^{(i)}).
\]
Thus \(u_*\) is a ground state. The ground state may coincide with one of the localized states.
\end{Theorem}

For example, fix \(-mc^2<V_0<V_\infty<mc^2\). On each terminal edge, set \(V=V_0\) on \([0,R_i]\), let \(V\) increase continuously and strictly from \(V_0\) to \(V_\infty\) on \([R_i,R_i+1]\), and put \(V=V_\infty\) on \([R_i+1,R_i+2]\) and on \(\mathcal G_0\). Then \(Z_i=[0,R_i]\), the neighborhoods in \((V_2)\) may be taken with \(\rho_0=1\), and the potential outside their union equals \(V_\infty\).

The barycenter map inspired by \cite{MR1409663} records the minimum component associated with each singleton solution. Existence of these solutions follows from the local construction. Ground-state attainment uses the energy bound
\[
0<\ell\le I(u^{(i)})<d_{V_\infty}^{\infty}.
\]
The proof of Theorem~\ref{Thm1.2} verifies the autonomous gap at \(V_\infty\) by applying the terminal construction to the constant potential \(V_\infty\). Proposition~\ref{Pro3.23}, based on Ekeland's variational principle and the compactness theorem in Lemma~\ref{Lem3.20}, then gives a ground state.

The variational arguments use the generalized Nehari method of Szulkin and Weth~\cite{MR2557725,MR2768820}, applied to the spectral decomposition of the Dirac operator with Kirchhoff-type vertex conditions. The dependence of the autonomous ground level on the constant potential is described by Propositions~\ref{Pro3.24} and \ref{Pro3.27}.

This article is organized as follows. In Section~\ref{Sec2}, we introduce the function space and the generalized Nehari reduction. In Section~\ref{Sec3}, we study compactness, ground states and the half-line problem. In Section~\ref{Sec4}, we prove the main results.

\textbf{Notation.} We use the following notation.
\begin{itemize}
\item[$\bullet$] For \(q\in(1,\infty)\), \(q'\) denotes the conjugate exponent of \(q\), that is, \(q'=\frac{q}{q-1}\).
\item[$\bullet$] For $t\in[1,\infty)$, $\|\cdot\|_{t}$ denotes the norm in $L^{t}(\mathcal G;\mathbb C^{2})$.
\item[$\bullet$] \(C\) and \(C_{i}\) denote possibly different positive constants whose values are not relevant.
\item[$\bullet$] If \(A \subset \mathcal{G}\) is a measurable set, we denote by \(|A|\) its Lebesgue measure.

\end{itemize}

\section{Variational framework}\label{Sec2}

In this section we introduce the metric graph $\mathcal{G}$, the Dirac operator with Kirchhoff-type vertex conditions, and the variational setting used throughout the paper. The proofs of the basic functional-analytic facts can be found, for instance, in \cite{MR3013208,MR3934110}.

\subsection{Quantum graphs and the Dirac operator}

A graph $\mathcal{G}=(\mathcal V,\mathcal E)$ consists of a finite or countably infinite set of vertices $\mathcal V$ and a set $\mathcal E$ of edges connecting the vertices. Two vertices $\mathrm{u}$ and $\mathrm{v}$ are adjacent (denoted $\mathrm{u}\sim\mathrm{v}$) if there is an edge connecting them. A quantum graph $\mathcal{G}$ is a connected multigraph, where multiple edges and loops are allowed. Each edge is identified either with a finite interval or with a half line in $\mathbb R$, and the edges are joined at their endpoints (the vertices of $\mathcal{G}$) according to the topology of the graph.

Unbounded edges are identified with copies of $\mathbb R^{+}=[0,\infty)$ and are called half lines, while bounded edges $e\in\mathcal E$ are identified with closed bounded intervals $I_e=[0,\ell_e]$ with $\ell_e>0$. On each edge $I_e$ we fix a coordinate $x_e$, with arbitrary orientation if $I_e$ is bounded, and with the natural orientation in the case of a half line. In this way $\mathcal{G}$ becomes a locally compact metric space equipped with the path metric induced by the shortest distance along the edges. A metric graph with finitely many edges and containing no half lines is compact.

\begin{Definition}\label{Def2.1}
We define the compact core $\mathcal{K}$ of $\mathcal{G}$ as the metric subgraph consisting of all its bounded edges.

\end{Definition}

\begin{Definition}\label{Def2.2}
A connected noncompact metric graph $\mathcal G$ is called a generalized star graph if:
\begin{enumerate}
\item $\mathcal K$ defined in Definition~\ref{Def2.1} is nonempty and compact;
\item $\mathcal G$ has finitely many half lines. More precisely, there exist $M\in\mathbb N$ and edges $e_1,\dots,e_M$ with pairwise disjoint interiors such that each $e_j$ is isometric to $[0,\infty)$ and
\[
\mathcal G=\mathcal K\cup \bigcup_{j=1}^M e_j, \qquad \mathcal K\cap e_j=\{v_j\},
\]
where $v_j\in\mathcal V$ is the attachment vertex of $e_j$ to the compact core.
\end{enumerate}
\end{Definition}

\begin{Remark}\label{Rem2.3}
The set $\mathcal G\setminus\mathcal K$ has exactly $M$ connected components, each isometric to $(0,\infty)$. The closure of each component in $\mathcal G$ is isometric to $[0,\infty)$. Thus the noncompactness of $\mathcal G$ is due to the finitely many half lines attached to $\mathcal K$.
\end{Remark}

Figure~\ref{fig:gen-star-examples} shows examples of these graphs; the first panel includes the terminal edges used in Theorem~\ref{Thm1.1}.

\begin{figure}[t]
\centering
\begin{subfigure}[t]{0.48\textwidth}
\centering
\begin{tikzpicture}[scale=0.70, line width=0.8pt]
\fill (-2.0,0) circle (2pt) node[above] {$z_1$};
\fill (0,0) circle (2pt) node[above] {$v_1$};
\fill (1.5,0) circle (2pt) node[above] {$v_2$};
\fill (3.5,0) circle (2pt) node[above] {$z_2$};
\draw[line width=1.7pt] (-2.0,0)--(-0.5,0);
\draw (-0.5,0)--(0,0)--(1.5,0)--(2.0,0);
\draw[line width=1.7pt] (2.0,0)--(3.5,0);
\draw (-0.5,-0.10)--(-0.5,0.10);
\draw (2.0,-0.10)--(2.0,0.10);
\node at (-1.1,-0.35) {$p_1$};
\node at (2.6,-0.35) {$p_2$};
\node at (0.75,0.35) {$\mathcal G_0$};
\draw[-{Latex[length=3mm]}] (0,0)--(0,-1.6) node[below] {$e_1$};
\draw[-{Latex[length=3mm]}] (1.5,0)--(1.5,-1.6) node[below] {$e_2$};
\end{tikzpicture}
\caption{Two terminal edges. Thick segments represent the constant-potential intervals.}
\end{subfigure}
\hfill
\begin{subfigure}[t]{0.48\textwidth}
\centering
\begin{tikzpicture}[scale=1, line width=0.9pt]
\fill (0,0) circle (2pt);
\fill (2,0) circle (2pt);
\fill (2,2) circle (2pt);
\fill (0,2) circle (2pt);
\draw (0,0)--(2,0)--(2,2)--(0,2)--(0,0);
\node at (1,1) {$\mathcal K$};
\fill (2,0) circle (2pt) node[below right] {$v_1$};
\draw[-{Latex[length=3mm]}] (2,0)--(3.5,-0.7) node[right] {$e_1$};
\fill (0,2) circle (2pt) node[above left] {$v_2$};
\draw[-{Latex[length=3mm]}] (0,2)--(-1.5,2.7) node[left] {$e_2$};
\end{tikzpicture}
\caption{Core is a loop (cycle) with two half lines attached.}
\end{subfigure}

\medskip

\begin{subfigure}[t]{0.48\textwidth}
\centering
\begin{tikzpicture}[scale=1, line width=0.9pt]
\fill (0,0) circle (2pt) node[left] {$v_0$};
\fill (1.7,0.7) circle (2pt) node[right] {$v_1$};
\fill (1.7,-0.7) circle (2pt) node[right] {$v_2$};
\fill (3.2,0) circle (2pt) node[right] {$v_3$};
\draw (0,0)--(1.7,0.7);
\draw (0,0)--(1.7,-0.7);
\draw (1.7,0.7)--(3.2,0);
\node at (1.2,0.15) {$\mathcal K$};
\draw[-{Latex[length=3mm]}] (1.7,-0.7)--(3.4,-1.8) node[right] {$e_1$};
\draw[-{Latex[length=3mm]}] (3.2,0)--(4.8,0.2) node[right] {$e_2$};
\draw[-{Latex[length=3mm]}] (0,0)--(-1.6,0.2) node[left] {$e_3$};
\end{tikzpicture}
\caption{Core is a bounded tree with three half lines attached.}
\end{subfigure}
\hfill
\begin{subfigure}[t]{0.48\textwidth}
\centering
\begin{tikzpicture}[scale=1, line width=0.9pt]
\fill (0,0) circle (2pt) node[left] {$v_1$};
\fill (2.6,0) circle (2pt) node[right] {$v_2$};
\draw (0,0) to[out=40,in=140] (2.6,0);
\draw (0,0)--(2.6,0);
\draw (0,0) to[out=-40,in=-140] (2.6,0);
\node at (1.3,0.55) {$\mathcal K$};
\draw[-{Latex[length=3mm]}] (0,0)--(-1.6,0) node[left] {$e_1$};
\draw[-{Latex[length=3mm]}] (2.6,0)--(4.2,0) node[right] {$e_2$};
\fill (1.3,0) circle (2pt);
\draw[-{Latex[length=3mm]}] (1.3,0)--(1.3,-1.6) node[below] {$e_3$};
\end{tikzpicture}
\caption{Core has multiple edges; three half lines are attached.}
\end{subfigure}
\caption{Examples of generalized star graphs. Panel (a) illustrates the terminal geometry in Theorem~\ref{Thm1.1}; the other panels illustrate the general definition.}
\label{fig:gen-star-examples}
\end{figure}

A function $u:\mathcal{G}\to\mathbb{C}$ can be regarded as a family $\{u_e\}_{e\in\mathcal E}$, where $u_e:I_e\to\mathbb{C}$ is the restriction of $u$ to the edge $I_e$. The usual $L^p$ spaces are defined over $\mathcal{G}$ in the natural way, with norm
\[
\|u\|_{L^p(\mathcal{G})}^p \;=\; \sum_{e\in\mathcal E} \|u_e\|_{L^p(I_e)}^p ,
\]
and we use the standard space $H^1(\mathcal{G})$ of continuous functions $u:\mathcal{G}\to\mathbb{C}$ such that $u_e\in H^1(I_e;\mathbb{C})$ for every $e\in\mathcal E$, endowed with the norm
\[
\|u\|_{H^1(\mathcal{G})}^2 \;=\; \|u'\|_{L^2(\mathcal{G})}^2 + \|u\|_{L^2(\mathcal{G})}^2 .
\]

A spinor $u=(u_{1},u_{2})^{T}:\mathcal{G}\to\mathbb{C}^{2}$ is a family of two-component functions
\[
u_{e}=\begin{pmatrix}u_{e}^{1}\\u_{e}^{2}\end{pmatrix} : I_{e}\to\mathbb{C}^{2} \quad \text{for all } e\in \mathcal{E}.
\]
We set
\[
L^{p}(\mathcal{G},\mathbb{C}^{2}) \;=\; \bigoplus_{e\in \mathcal{E}} L^{p}(I_{e};\mathbb{C}^{2}), \qquad \|u\|_{L^{p}(\mathcal{G},\mathbb{C}^{2})}^{p} \;=\; \sum_{e\in\mathcal{E}} \|u_{e}\|_{L^{p}(I_{e};\mathbb{C}^{2})}^{p},
\]
\[
H^{1}(\mathcal{G},\mathbb{C}^{2}) \;=\; \bigoplus_{e\in\mathcal{E}} H^{1}(I_{e};\mathbb{C}^{2}), \qquad \|u\|_{H^{1}(\mathcal{G},\mathbb{C}^{2})}^{2} \;=\; \sum_{e\in\mathcal{E}} \|u_{e}\|_{H^{1}(I_{e};\mathbb{C}^{2})}^{2}.
\]
For $u\in H^{1}(\mathcal{G},\mathbb{C}^{2})$, the traces $u_{e}(\mathrm v)$ at vertices are well-defined edgewise; no vertex condition is imposed at this stage.

Let
\[
\mathcal{D}=-i c\,\sigma_1\,\frac{d}{dx} + m c^{2} \sigma_3
\]
be the Dirac operator on edges. Then equation \eqref{eq1.1} can be written as
\begin{equation}\label{eq2.1}
\mathcal{D} u + V(x)\,u = f(|u|)\,u \quad \text{in } \mathcal{G}.
\end{equation}

The expression of $\mathcal{D}$ is purely formal unless one prescribes vertex conditions, since the derivative is defined only in the interior of each edge. We restrict ourselves to Kirchhoff-type vertex conditions introduced in \cite{MR3934110}, which represent the free case for the Dirac operator on graphs.

\begin{Definition}\label{Def2.4}
Let $\mathcal{G}$ be a quantum graph and let $m,c>0$. We define the Dirac operator with Kirchhoff-type vertex conditions $\mathcal{D}: L^{2}(\mathcal{G},\mathbb{C}^{2}) \to L^{2}(\mathcal{G},\mathbb{C}^{2})$ by
\[
\mathcal{D}|_{I_{e}} u \;=\; \mathcal{D}_{e} u_{e} \;=\; - i c\,\sigma_1\, u_{e}' + m c^{2} \sigma_3 u_{e}, \quad \text{for all } e \in \mathcal{E},
\]
with domain
\[
\operatorname{dom}(\mathcal{D}) \;=\; \Big\{ u \in H^{1}(\mathcal{G},\mathbb{C}^{2}) \,:\, u \text{ satisfies } \eqref{eq2.2} \text{ and } \eqref{eq2.3} \Big\},
\]
where for every vertex $\mathrm{v}\in\mathcal V$
\begin{equation}\label{eq2.2}
u_{e}^{1}(\mathrm{v}) \;=\; u_{e'}^{1}(\mathrm{v}) \quad \text{for all } e,e' \succ \mathrm{v},
\end{equation}
\begin{equation}\label{eq2.3}
\sum_{e \succ \mathrm{v}} \big(u_{e}^{2}(\mathrm{v})\big)_{\pm} \;=\; 0 .
\end{equation}
Here $e\succ \mathrm{v}$ means that the edge $e$ is incident to the vertex $\mathrm{v}$, and $\big(u_{e}^{2}(\mathrm{v})\big)_{\pm}$ stands for $u_{e}^{2}(0)$ or $-u_{e}^{2}(\ell_{e})$ according to whether $x_{e}$ equals $0$ or $\ell_{e}$ at $\mathrm{v}$. All vertex conditions are understood in terms of incident edge endpoints. In particular, a loop incident to a vertex at both endpoints contributes twice to the sum in \eqref{eq2.3}.
\end{Definition}

\begin{Remark}\label{Rem2.5}
The operator $\mathcal{D}$ depends on $m$ and $c$. Unless otherwise stated, we omit this dependence.
\end{Remark}

\begin{samepage}
The following properties were established in \cite{MR3934110}.
\begin{Proposition}\label{Pro2.6}
The Dirac operator $\mathcal{D}$ introduced in Definition \ref{Def2.4} is self-adjoint on $L^{2}(\mathcal{G},\mathbb{C}^{2})$. Moreover, its essential spectrum is
\[
\sigma_{\mathrm{ess}}(\mathcal{D}) \;=\; \left(-\infty,-m c^{2}\right] \cup \left[m c^{2},+\infty\right).
\]
In particular,
\[
\sigma(\mathcal{D})\cap(-m c^{2},m c^{2})=\emptyset .
\]
\end{Proposition}
\end{samepage}

For completeness, the absence of spectrum in the open gap also follows directly from the vertex conditions. Integration by parts gives
\[
\|\mathcal Du\|_2^2=c^2\|u'\|_2^2+m^2c^4\|u\|_2^2 \qquad (u\in\operatorname{dom}\mathcal D).
\]
The boundary terms cancel because the first component is continuous and the signed sum of the second components is zero at each vertex. Self-adjointness then gives $\mathcal D^2\ge m^2c^4$ and hence the stated spectral gap.

\subsection{Form domain and the energy functional}

We define the nonnegative quadratic form associated with $|\mathcal D|$ through the spectral theorem. For the spectral measure $\mu_{u}^{\mathcal{D}}$ we set
\[
\operatorname{dom}(\mathcal{Q}_{\mathcal{D}}) = \Big\{\,u\in L^{2}(\mathcal{G},\mathbb{C}^{2})\,:\,\int_{\sigma(\mathcal{D})} |v|\, d\mu_{u}^{\mathcal{D}}(v)<\infty\,\Big\}, \qquad \mathcal{Q}_{\mathcal{D}}(u) = \int_{\sigma(\mathcal{D})} |v|\, d\mu_{u}^{\mathcal{D}}(v).
\]
Equivalently, $\mathcal{Q}_{\mathcal{D}}(u)=\|\,|\mathcal{D}|^{1/2}u\|_{L^{2}}^{2}$.

An alternative and convenient description of the form domain uses real interpolation theory \cite{MR2424078,ameur2019interpolation}. Define
\[
Y = (\,L^{2}(\mathcal{G},\mathbb{C}^{2}),\,\operatorname{dom}(\mathcal{D})\,)_{1/2,2}.
\]
Then $Y$ is continuously embedded into
\[
H^{1/2}(\mathcal{G},\mathbb{C}^{2}) = \bigoplus_{e\in \mathcal{E}} H^{1/2}(I_{e};\mathbb{C}^{2}),
\]
and by Sobolev embeddings one has
\[
Y \hookrightarrow L^{p}(\mathcal{G},\mathbb{C}^{2}) \quad \text{for all } p\in[2,\infty).
\]
Moreover, since $\mathcal K$ is compact, the embedding $Y\hookrightarrow L^{p}(\mathcal{K},\mathbb{C}^{2})$ is compact for every $p\in[2,\infty)$. On the other hand,
\[
\operatorname{dom}(\mathcal{Q}_{\mathcal{D}})=Y.
\]

By Proposition~\ref{Pro2.6} we can decompose the form domain as the orthogonal sum of the positive and negative spectral subspaces,
\[
Y = Y^{+} \oplus Y^{-}.
\]
Accordingly, every $u\in Y$ can be written as $u=u^{+}+u^{-}$, where $u^{\pm}=P^{\pm}u$ and $P^{\pm}$ are the orthogonal projectors onto $Y^{\pm}$.

We use the equivalent norm
\[
\|u\| = \||\mathcal{D}|^{1/2}u\|_{L^{2}} \quad \text{for all } u \in Y,
\]
and denote by $(\cdot,\cdot)$ the inner product on $Y$ given by
\[
(u,v) = \operatorname{Re}\langle |\mathcal{D}|^{1/2}u, |\mathcal{D}|^{1/2}v \rangle_{L^{2}} .
\]

We define the energy functional associated with \eqref{eq2.1} on $Y$ by
\[
I(u) = \frac{1}{2}\big(\|u^{+}\|^{2}-\|u^{-}\|^{2}\big) + \frac{1}{2}\int_{\mathcal{G}} V(x)|u|^{2}\,dx - \int_{\mathcal{G}} F(|u|)\,dx , \quad u\in Y,
\]
where $F(t)=\int_{0}^{t} f(s)s\,ds$. Then $I\in C^{1}(Y,\mathbb{R})$ and
\[
I^{\prime}(u)[v] = (u^{+}, v^{+}) - (u^{-}, v^{-}) + \operatorname{Re}\int_{\mathcal{G}} V(x)\, u \cdot \bar{v}\,dx - \operatorname{Re}\int_{\mathcal{G}} f(|u|)\, u \cdot \bar{v}\,dx \quad \text{for all } u,v\in Y.
\]
If $v\in\operatorname{dom}(\mathcal D)$, then the first two terms can be equivalently written as $(u^{+}, v^{+}) - (u^{-}, v^{-}) = \operatorname{Re}\int_{\mathcal G}\langle u,\mathcal D v\rangle\,dx$.

\begin{Lemma}\label{Lem2.7}
Let $m>0$ and $c>0$. Since $\sigma(\mathcal{D})\cap(-mc^{2},mc^{2})=\emptyset$, one has
\[
mc^{2}\,\|u\|_{L^{2}}^{2}\le \|u\|^{2} \quad \text{for all } u\in Y .
\]
\end{Lemma}

\begin{proof}
By the spectral theorem, for $u\in Y=\operatorname{dom}(|\mathcal D|^{1/2})$ one has
\[
\|u\|^{2} \;=\; \int_{\sigma(\mathcal D)} |v|\, d\mu_{u}^{\mathcal D}(v), \qquad \|u\|_{L^{2}}^{2} \;=\; \int_{\sigma(\mathcal D)} 1 \, d\mu_{u}^{\mathcal D}(v).
\]
Since $|v|\ge mc^{2}$ for all $v\in\sigma(\mathcal D)$, it follows that
\[
\|u\|^{2} \ge mc^{2}\int 1\, d\mu_{u}^{\mathcal D}(v) = mc^{2}\,\|u\|_{L^{2}}^{2}.
\]
\end{proof}

\begin{Lemma}\label{Lem2.8}
Under assumptions $(f_{1})$-$(f_{2})$, there exist $A,B>0$ and $\theta>2$ such that
\[
F(|t|)\ge A\,|t|^{\theta}-B\,|t|^{2} \quad \text{for all } t\ge 0 .
\]
\end{Lemma}

\begin{proof}
By $(f_{2})$ there is $\theta>2$ such that for all $t>0$,
\[
\theta\,F(t)\le f(t)\,t^{2}=F'(t)\,t .
\]
Hence the function $G(t)=\frac{F(t)}{t^{\theta}}$ satisfies
\[
G'(t)=\frac{F'(t)\,t-\theta F(t)}{t^{\theta+1}}\ge 0 \quad \text{for } t>0,
\]
so $G$ is nondecreasing on $(0,\infty)$. In particular, for all $t\ge 1$,
\[
F(t)\ge F(1)\,t^{\theta}.
\]
Set $A=F(1)>0$. For $0\le t\le 1$, since $\theta>2$ one has $t^{\theta}\le t^{2}$, so choosing $B\ge A$ gives
\[
A\,t^{\theta}-B\,t^{2}\le 0\le F(t).
\]
Replacing $t$ by $|t|$ yields the claim.
\end{proof}

A spinor \(u\in Y\) is called a weak solution of \eqref{eq2.1} if
\[
\operatorname{Re}\int_{\mathcal G}u\cdot\overline{\mathcal D\varphi}\,dx+\operatorname{Re}\int_{\mathcal G}Vu\cdot\overline\varphi\,dx =\operatorname{Re}\int_{\mathcal G}f(|u|)u\cdot\overline\varphi\,dx
\]
for every \(\varphi\in\operatorname{dom}\mathcal D\). All integrals are finite by \((f_1)\) and the embeddings of \(Y\). A nonzero weak solution is a bound state.

\begin{Proposition}\label{Pro2.9}
A spinor is a weak solution of the NLDE \eqref{eq2.1} if and only if it is a critical point of $I$.
\end{Proposition}

\begin{proof}
If $u$ is a weak solution, then the derivative formula above yields $I'(u)=0$.

Conversely, assume that $u\in Y$ is a critical point, so that $I'(u)[\varphi]=0$ for all $\varphi\in Y$. In particular, for every $\varphi\in\operatorname{dom}(\mathcal D)\subset Y$ we have
\[
0=I'(u)[\varphi]=\operatorname{Re}\int_{\mathcal G}u\cdot\overline{\mathcal D\varphi}\,dx+\operatorname{Re}\int_{\mathcal G}V(x)u\cdot\overline\varphi\,dx-\operatorname{Re}\int_{\mathcal G}f(|u|)u\cdot\overline\varphi\,dx.
\]
Testing also with $i\varphi$ gives the corresponding complex identity.

Fix an edge $e\in\mathcal E$ and take
\[
\varphi=\begin{pmatrix}\varphi^1\\0\end{pmatrix},\qquad \varphi^1\in C_0^\infty(I_e),\quad\varphi^1\not\equiv0.
\]
Since $\varphi$ is compactly supported in the interior of $I_e$, we have $\varphi\in\operatorname{dom}(\mathcal D)$ and all vertex traces vanish. Writing $u_e=(u_e^1,u_e^2)^T$ and $\varphi_e=(\varphi_e^1,0)^T$, the complex identity gives
\[
\int_{I_e}u_e\cdot\overline{(\mathcal D+V(x))\varphi_e}\,dx=\int_{I_e}f(|u_e|)u_e\cdot\overline{\varphi_e}\,dx.
\]
A direct computation gives
\[
\int_{I_e}\bigl((mc^2+V)u_e^1\overline{\varphi_e^1}+ic\,u_e^2(\overline{\varphi_e^1})'\bigr)\,dx=\int_{I_e}f(|u_e|)u_e^1\overline{\varphi_e^1}\,dx,
\]
hence
\[
ic\int_{I_e}u_e^2(\overline{\varphi_e^1})'\,dx=\int_{I_e}\bigl(f(|u_e|)-mc^2-V(x)\bigr)u_e^1\overline{\varphi_e^1}\,dx.
\]
By $(f_1)$ and the embeddings of $Y$ into $L^2\cap L^{2p-2}$, we have $\bigl(f(|u|)-mc^2-V\bigr)u^1\in L^2(I_e)$. Hence the distributional derivative of $u_e^2$ belongs to $L^2(I_e)$ and $u_e^2\in H^1(I_e)$. An integration by parts on $I_e$ then yields the first component of \eqref{eq2.1} on $I_e$. Exchanging the roles of the components by taking $\varphi=(0,\varphi^2)^T$ with $\varphi^2\in C_0^\infty(I_e)$, we similarly obtain $u_e^1\in H^1(I_e)$ and the second component of \eqref{eq2.1} on $I_e$. Since the graph has finitely many edges, $u\in H^1(\mathcal G,\mathbb C^2)$.

It remains to verify the vertex conditions \eqref{eq2.2} and \eqref{eq2.3}. Fix a vertex $\mathrm v$. Choose a compactly supported test spinor
\[
\varphi=\begin{pmatrix}\varphi^1\\0\end{pmatrix}\in\operatorname{dom}(\mathcal D),\qquad \varphi^1(\mathrm v)=1,\qquad\varphi(\mathrm v')=0\quad\text{for all vertices }\mathrm v'\ne\mathrm v.
\]
Applying the weak identity and integrating by parts on each incident edge, the boundary terms at other vertices vanish and we obtain
\[
\sum_{e\succ\mathrm v}u_e^2(\mathrm v)_\pm\,\overline{\varphi_e^1(\mathrm v)}=0.
\]
By continuity of $\varphi^1$ at $\mathrm v$ for functions in $\operatorname{dom}(\mathcal D)$, all $\varphi_e^1(\mathrm v)$ equal $\varphi^1(\mathrm v)=1$, hence
\[
\sum_{e\succ\mathrm v}u_e^2(\mathrm v)_\pm=0,
\]
that is \eqref{eq2.3}.

Now let $\mathrm v$ be a vertex of degree at least $2$. As in Definition~\ref{Def2.4}, incidences at a vertex are counted separately. Choose two incident edge endpoints at $\mathrm v$, indexed by $e_1,e_2$, and take
\[
\varphi=\begin{pmatrix}0\\\varphi^2\end{pmatrix}\in\operatorname{dom}(\mathcal D),\qquad \varphi_{e_1}^2(\mathrm v)_\pm=1,\quad\varphi_{e_2}^2(\mathrm v)_\pm=-1,
\]
with all other incident traces zero and with compact support near $\mathrm v$. This is admissible because $\sum_{e\succ\mathrm v}\varphi_e^2(\mathrm v)_\pm=0$. Using the weak identity and integrating by parts as before yields
\[
u_{e_1}^1(\mathrm v)\overline{\varphi_{e_1}^2(\mathrm v)_\pm}+u_{e_2}^1(\mathrm v)\overline{\varphi_{e_2}^2(\mathrm v)_\pm}=0,
\]
that is $u_{e_1}^1(\mathrm v)=u_{e_2}^1(\mathrm v)$. Since the two incidences are arbitrary, we obtain \eqref{eq2.2}, including when the two endpoints belong to a loop. Repeating the argument at every vertex completes the proof.
\end{proof}

\begin{Lemma}\label{Lem2.10}
Let $\mathcal G$ be a generalized star graph, $r>0$ and $2\le q<\infty$. Assume that $(u_n)$ is bounded in $Y$ and
\[
\sup_{y\in\mathcal G}\int_{B(y,r)} |u_n|^{q}\,dx \to 0 \quad \text{as } n\to\infty,
\]
where $B(y,r)=\{x\in\mathcal G:\ d(x,y)<r\}$. Then $u_n\to 0$ in $L^{p}(\mathcal G,\mathbb C^{2})$ for every $2<p<\infty$.
\end{Lemma}

\begin{proof}
Fix $r>0$. Since $\mathcal G$ is a generalized star graph, there exists a constant $C_r>0$ such that
\[
|B(y,r)|\le C_r\qquad\text{for all }y\in\mathcal G.
\]
By H\"older's inequality, for every $y\in\mathcal G$,
\[
\int_{B(y,r)}|u_n|^2\,dx\le |B(y,r)|^{1-2/q}\left(\int_{B(y,r)}|u_n|^q\,dx\right)^{2/q}\le C_r^{1-2/q}\left(\int_{B(y,r)}|u_n|^q\,dx\right)^{2/q}.
\]
Hence
\[
\sup_{y\in\mathcal G}\int_{B(y,r)}|u_n|^2\,dx\longrightarrow0\quad\text{as }n\to\infty.
\]
A finite cover of the compact core gives $u_n\to0$ in $L^2(\mathcal K)$. Compactness of the restriction $Y\to L^4(\mathcal K)$ then gives $u_n\to0$ in $L^4(\mathcal K)$.

On an external edge, partition $[0,\infty)$ into intervals $I_j=[jr,(j+1)r]$. The one-dimensional interpolation and Sobolev inequalities give
\begin{equation}\label{eq2.4}
\|v\|_{L^4(I_j)}^4 \le C_r\|v\|_{H^{1/2}(I_j)}^2\|v\|_{L^2(I_j)}^2.
\end{equation}
The embedding $H^{1/4}(I_j)\hookrightarrow L^4(I_j)$ and interpolation between $L^2(I_j)$ and $H^{1/2}(I_j)$ prove \eqref{eq2.4}. The constant depends only on $r$. Each interval $I_j$ lies, up to its endpoints, in a ball of radius $r$. Moreover, the Gagliardo seminorm satisfies
\[
\sum_j\|v\|_{H^{1/2}(I_j)}^2 \le \|v\|_{H^{1/2}([0,\infty))}^2.
\]
Summing \eqref{eq2.4} over the intervals and the finitely many external edges yields
\[
\|u_n\|_{L^4(\mathcal G\setminus\mathcal K)}^4 \le C\left(\sup_{y\in\mathcal G}\int_{B(y,r)}|u_n|^2\,dx\right)\|u_n\|^2 \longrightarrow0.
\]
Consequently $u_n\to0$ in $L^4(\mathcal G)$. For $2<p<4$, interpolate with the bounded $L^2$ norm. For $p>4$, choose a finite $q>p$ and interpolate with the bounded $L^q$ norm supplied by the embedding of $Y$. This proves the assertion for every $2<p<\infty$.
\end{proof}

\begin{Lemma}\label{Lem2.11}
Let $V\in L^\infty(\mathcal G,\mathbb R)$ satisfy $\|V\|_\infty<mc^2$, and assume $(f_1)$--$(f_2)$. Every Palais--Smale sequence for $I$ at a finite level is bounded in $Y$. In addition, every sequence satisfying
\[
I(u_n)\le C,\qquad I'(u_n)[u_n]=I'(u_n)[u_n^+-u_n^-]=0
\]
is bounded in $Y$.
\end{Lemma}
\begin{proof}
Put $H(s)=\tfrac12f(s)s^2-F(s)$. Assumption $(f_2)$ gives
\[
H(s)\ge\frac{\theta-2}{2\theta}f(s)s^2\ge0.
\]
Let $(u_n)$ be a Palais--Smale sequence, set $s_n=\|u_n\|$ and $\varepsilon_n=\|I'(u_n)\|_{Y^*}\to0$. Before boundedness has been established, the available estimate is
\begin{equation}\label{eq2.5}
\int_{\mathcal G}H(|u_n|)\,dx =I(u_n)-\frac12 I'(u_n)[u_n] \le C+\frac12\varepsilon_ns_n.
\end{equation}
Choose $r=p/(p-1)\in(1,2)$. For every $\tau>0$, continuity at zero and the growth of $f$ give a decomposition $f(s)s=g_\tau(s)+j_\tau(s)$ such that
\[
|g_\tau(s)|\le\tau s, \qquad |j_\tau(s)|^r\le C_\tau f(s)s^2.
\]
To verify the second estimate, take $j_\tau(s)=f(s)s\,\mathbf1_{[a_\tau,\infty)}(s)$, where $f\le\tau$ on $[0,a_\tau]$. The factor $f(s)^{r-1}s^{r-2}$ is bounded for $s\ge a_\tau$ because $(p-2)(r-1)+r-2=0$.

Set $\kappa=1-\|V\|_\infty/(mc^2)>0$. Testing with $u_n^+-u_n^-$, applying H\"older's inequality, and using \eqref{eq2.5}, we obtain
\[
\kappa s_n^2 \le\varepsilon_ns_n+\frac{\tau}{mc^2}s_n^2 +C_\tau s_n(C+\varepsilon_ns_n)^{1/r}.
\]
Fix $\tau<\kappa mc^2/2$. If a subsequence satisfied $s_n\to\infty$, division by $s_n^2$ would give
\[
\frac\kappa2 \le\frac{\varepsilon_n}{s_n} +C_\tau\frac{(C+\varepsilon_ns_n)^{1/r}}{s_n}\longrightarrow0,
\]
since $1/r<1$ and $(\varepsilon_n)$ is bounded. This is a contradiction. Under the two exact derivative identities in the second assertion, the same estimates hold with the derivative terms equal to zero. They again imply boundedness.
\end{proof}

\begin{Lemma}\label{Lem2.12}
Let $\eta$ be a bounded, real-valued Lipschitz function on $\mathcal G$ with bounded edgewise derivative. Multiplication by $\eta$ is bounded on $Y$. If $\eta_1,\ldots,\eta_N$ are such functions and $\sum_j\eta_j^2=1$, then
\begin{equation}\label{eq2.6}
\|u^+\|^2-\|u^-\|^2 =\sum_{j=1}^N\bigl(\|(\eta_ju)^+\|^2-\|(\eta_ju)^-\|^2\bigr) \qquad(u\in Y).
\end{equation}
\end{Lemma}
\begin{proof}
The common vertex value of $\eta$ preserves both vertex conditions. On $\operatorname{dom}\mathcal D$,
\[
\mathcal D(\eta u)=\eta\mathcal Du-ic\sigma_1\eta'u.
\]
This proves boundedness on the operator domain; boundedness on $L^2$ is immediate. Real interpolation gives boundedness on $Y$. For $u\in\operatorname{dom}\mathcal D$, the matrix $\sigma_1$ is Hermitian, so
\[
\operatorname{Re}\langle\mathcal D(\eta u),\eta u\rangle_2 =\operatorname{Re}\int_{\mathcal G}\eta^2\mathcal Du\cdot\overline u\,dx.
\]
Sum over $j$ and use $\sum_j\eta_j^2=1$. Density of $\operatorname{dom}\mathcal D$ in $Y$ gives \eqref{eq2.6}.
\end{proof}

\subsection{Generalized Nehari reduction}\label{Subsec2.3}

Let \(V\in L^\infty(\mathcal G,\mathbb R)\) satisfy \(\|V\|_\infty<mc^2\), and assume \((f_1)\)--\((f_2)\). For the functional \(I\) defined above, set
\[
\mathcal M=\{u\in Y\setminus Y^-: I'(u)[u]=0,\ I'(u)[v]=0\text{ for every }v\in Y^-\},
\]
and define
\[
S^+=\{w\in Y^+:\|w\|=1\},\qquad \ell=\inf_{u\in\mathcal M}I(u).
\]
For \(u\in Y\setminus Y^-\), write
\[
Y(u)=Y^-\oplus\mathbb Ru,\qquad \hat Y(u)=Y^-\oplus[0,\infty)u.
\]
We first prove the reduction for variable potentials. The autonomous case will follow by taking \(V\equiv\lambda\).

The following pointwise inequality is used on the fibres of \(I\). Its proof follows the argument in \cite[Lemma 3.3]{MR3567492}.

\begin{Lemma}\label{Lem2.13}
Let $t\ge 0$ and $u,v\in \mathbb{C}^{2}$. Then
\begin{equation}\label{eq2.7}
\operatorname{Re}\Big[f(|u|)\,u \cdot \overline{\Big(\frac{t^{2}}{2}\,u-\frac{1}{2}\,u+t\,v\Big)}\Big] + F(|u|) - F(|t u+v|) \;\le\; 0 .
\end{equation}
Moreover, if $u\ne 0$, equality in \eqref{eq2.7} holds if and only if $(t,v)=(1,0)$.
\end{Lemma}

\begin{proof}
The function $a\mapsto F(\sqrt a)$ is convex on $[0,\infty)$, since its derivative for $a>0$ is $f(\sqrt a)/2$. Therefore
\[
F(|tu+v|)-F(|u|) \ge\frac{f(|u|)}2\bigl(|tu+v|^2-|u|^2\bigr).
\]
Substitution into the left-hand side of \eqref{eq2.7} gives an upper bound $-\tfrac12f(|u|)|v|^2$. If $u=0$, the assertion reduces to $-F(|v|)\le0$. If $u\ne0$ and $v\ne0$, strict negativity follows from $f(|u|)>0$. If $u\ne0$ and $v=0$, strict monotonicity of $f$ gives strict convexity along the interval between $|u|^2$ and $t^2|u|^2$, so equality holds precisely when $t=1$.
\end{proof}

\begin{Proposition}\label{Pro2.14}
Let $V\in L^\infty(\mathcal G,\mathbb R)$ satisfy $\|V\|_\infty<mc^2$, and assume $(f_1)$--$(f_2)$. For every $w\in Y\setminus Y^-$, the restriction of $I$ to $\hat Y(w)$ has a unique global maximizer $\widehat m_V(w)$. This maximizer belongs to $\mathcal M$. The map
\[
m=\widehat m_V\big|_{Y^+\setminus\{0\}}:Y^+\setminus\{0\}\longrightarrow\mathcal M
\]
is continuous, and its restriction to $S^+$ is a homeomorphism with inverse $u\mapsto u^+/\|u^+\|$. Furthermore,
\[
0<\ell=\inf_{w\in S^+} I(m(w))=\inf_{u\in\mathcal M}I(u).
\]
\end{Proposition}

\begin{proof}
Write
\[
q_V(u,v)=(u^+,v^+)-(u^-,v^-)+\operatorname{Re}\int_{\mathcal G}Vu\cdot\overline v\,dx, \qquad \kappa=1-\frac{\|V\|_\infty}{mc^2}>0.
\]
The spectral gap gives $q_V(v,v)\le-\kappa\|v\|^2$ on $Y^-$ and $q_V(w,w)\ge\kappa\|w\|^2$ on $Y^+$. A sufficiently small sphere in $Y^+$ has strictly positive energy, uniformly in its direction.

Fix $w\in S^+$. The functional is nonpositive outside a sufficiently large ball in $\mathbb Rw\oplus Y^-$. Indeed, if $z_n=t_nw+v_n^-$, with $t_n\in\mathbb R$, had $I(z_n)>0$ and $\|z_n\|\to\infty$, then the estimate
\[
0<I(z_n)\le\frac{1+a}{2}\|z_n^+\|^2-\frac{1-a}{2}\|z_n^-\|^2, \qquad a=\frac{\|V\|_\infty}{mc^2},
\]
would keep $|t_n|/\|z_n\|$ away from zero. The normalized sequence has a nonzero weak limit because its positive part is a nonzero multiple of the fixed $w$. Local compactness and the superquadratic lower bound for $F$ then give
\[
\int_{\mathcal G}\frac{F(|z_n|)}{\|z_n\|^2}\,dx\longrightarrow\infty,
\]
contrary to $I(z_n)>0$. The same argument is uniform when $w$ varies over a compact subset of $S^+$.

For a bounded maximizing sequence $t_nw+v_n^-$, take $t_n\to t$ and $v_n^-\rightharpoonup v^-$. On $Y^-$, the form $-q_V$ is a positive definite equivalent Hilbert norm squared. The term $q_V(w,v_n^-)$ is weakly continuous. Also, $u\mapsto\int F(|u|)$ is convex and weakly lower semicontinuous. These facts give weak upper semicontinuity of $I$ on the fibre. The positive supremum is therefore attained with $t>0$, and the first-order conditions give an element of $\mathcal M$.

For $u\in\mathcal M$, $s\ge0$ and $\eta\in Y^-$, expanding the quadratic form and using the two Nehari identities gives
\begin{align*}
I(su+\eta)-I(u) &=\frac12q_V(\eta,\eta)\\
&\quad+\int_{\mathcal G}\left[\operatorname{Re}\left(f(|u|)u\cdot\overline{\frac{s^2-1}{2}u+s\eta}\right)+F(|u|)-F(|su+\eta|)\right]dx.
\end{align*}
Lemma~\ref{Lem2.13} makes this negative unless $(s,\eta)=(1,0)$. Thus the maximizer is unique.

To prove continuity, let $w_n\to w$ in $S^+$. Uniform fibre coercivity bounds $m(w_n)=t_nw_n+v_n^-$, and the uniform positive maximum keeps $t_n$ away from zero. After extracting a subsequence, let $t_n\to t>0$ and $v_n^-\rightharpoonup v^-$. Comparing the maximum with every fixed competitor $sw_n+\eta$ shows, by weak upper semicontinuity, that $tw+v^-=m(w)$ and that the maximum values converge. The positive and mixed quadratic terms converge. Weak lower semicontinuity of $-q_V(v_n^-,v_n^-)$ and $\int F(|m(w_n)|)$ then forces
\[
-q_V(v_n^-,v_n^-)\longrightarrow -q_V(v^-,v^-).
\]
Consequently $v_n^-\to v^-$ in $Y^-$ and $m(w_n)\to m(w)$ in $Y$. Uniqueness removes the subsequence. The positive lower bound for the maximum and the estimate above imply $\|m(w)^+\|\ge c_0>0$. The inverse $u\mapsto u^+/\|u^+\|$ is continuous, and the assertions about $\ell$ follow from the fibre characterization.
\end{proof}

\begin{Proposition}\label{Pro2.15}
\(I\) is coercive on \(\mathcal{M}\).
\end{Proposition}

\begin{proof}
For $u\in\mathcal M$, both $I'(u)[u]$ and $I'(u)[u^+-u^-]$ vanish. The second assertion of Lemma~\ref{Lem2.11} therefore bounds every sublevel of $I$ on $\mathcal M$.
\end{proof}

\begin{Lemma}\label{Lem2.16}
Define
\[
\hat{\Upsilon}: Y^{+}\setminus\{0\}\to \mathbb{R}, \qquad \hat{\Upsilon}(y)=I\big(m(y)\big).
\]
Then \(\hat{\Upsilon}\in C^{1}\big(Y^{+}\setminus\{0\},\mathbb{R}\big)\) and
\[
\hat{\Upsilon}'(y)[z] = \frac{\|m(y)^{+}\|}{\|y\|} \, I'\!\big(m(y)\big)[z], \qquad \forall\, y,z\in Y^{+},\; y\neq 0 .
\]
\end{Lemma}

\begin{proof}
The fibre comparison follows \cite[Proposition 2.9]{MR2557725}. Write \(m(y)=v_0^-+s_0y\), where \(s_0=\|m(y)^+\|/\|y\|\), and \(m(y+tz)=v_t^-+s_t(y+tz)\). Continuity of \(m\) gives \(s_t\to s_0\) and \(v_t^-\to v_0^-\). Maximality on the two fibres yields
\[
I(m(y)+ts_0z)-I(m(y))\le\hat\Upsilon(y+tz)-\hat\Upsilon(y)\le I(m(y+tz))-I(m(y+tz)-ts_tz).
\]
Dividing by \(t\), with the inequalities reversed when \(t<0\), and using \(I\in C^1\), we obtain
\[
\frac{\hat\Upsilon(y+tz)-\hat\Upsilon(y)}t\longrightarrow s_0I'(m(y))[z].
\]
The resulting bounded linear functional depends continuously on \(y\), so \(\hat\Upsilon\in C^1\) and the asserted formula holds.
\end{proof}

\begin{Corollary}\label{Cor2.17}
Let \(\Upsilon=\hat{\Upsilon}\mid_{S^{+}}\). Then:
\begin{enumerate}
\item[\textnormal{(i)}] \(\Upsilon\in C^{1}(S^{+})\) and
\[
\Upsilon'(w)[z] = \|m(w)^{+}\|\; I'\!\big(m(w)\big)[z], \qquad \forall\, z\in T_{w}S^{+} = \{v\in Y^{+}:\ (w,v)=0\}.
\]
In particular, \(w\in S^{+}\) is a critical point of \(\Upsilon\) on \(S^{+}\) if and only if \(u=m(w)\) is a critical point of \(I\) in \(Y\), and the critical values coincide.

\item[\textnormal{(ii)}] Let \(\beta\in\mathbb R\). A sequence \((w_n)\subset S^{+}\) is a \((PS)_\beta\) sequence for \(\Upsilon\) if and only if \(\big(m(w_n)\big)\subset\mathcal M\) is a \((PS)_\beta\) sequence for \(I\), namely
\[
I(m(w_n))\to \beta, \qquad I'(m(w_n))\to0 \quad\text{in }Y^* .
\]

\item[\textnormal{(iii)}] \(\Upsilon\) satisfies the \((PS)_\beta\) condition on \(S^{+}\) if and only if \(I\) satisfies the \((PS)_\beta\) condition on \(\mathcal M\), that is, every sequence \((u_n)\subset\mathcal M\) with
\[
I(u_n)\to \beta, \qquad I'(u_n)\to0 \quad\text{in }Y^*
\]
has a convergent subsequence in \(Y\).
\end{enumerate}
\end{Corollary}

\begin{proof}
We follow \cite[Corollary 2.10]{MR2557725}. Lemma~\ref{Lem2.16} gives the derivative formula. Set \(u=m(w)\). Since \(I'(u)\) vanishes on \(Y^-\oplus\mathbb Rw\), the orthogonal decomposition of \(Y\) gives
\[
\|\Upsilon'(w)\|_{(T_wS^+)^*}=\|u^+\|\,\|I'(u)\|_{Y^*}.
\]
Proposition~\ref{Pro2.14} gives \(\|u^+\|\ge c_0>0\). On each bounded energy sublevel, Proposition~\ref{Pro2.15} also gives an upper bound for \(\|u^+\|\). This proves (i) and (ii). The homeomorphism in Proposition~\ref{Pro2.14} then gives (iii).
\end{proof}

The same fibre arguments apply to the Dirac operator on \(\mathbb R\), using its spectral gap and the embeddings of \(H^{1/2}(\mathbb R,\mathbb C^2)\) into every finite \(L^q\) space with \(q\ge2\).

\section{Ground states and the half-line problem}\label{Sec3}

\subsection{Autonomous fibre properties}\label{Subsec3.1}

We study the existence of solutions to the nonlinear Dirac equation on a noncompact metric graph with a constant potential $\lambda$,
\begin{equation}\label{eq3.1}
\mathcal{D}u + \lambda u = f(|u|)\,u \quad \text{on } \mathcal{G},
\end{equation}
where $\mathcal{D}$ is the Dirac operator with Kirchhoff-type vertex conditions introduced in Definition~\ref{Def2.4}, $|\lambda|<mc^{2}$, and $f$ satisfies $(f_{1})$--$(f_{2})$.

As in Section~\ref{Sec2}, $\mathcal G$ is a generalized star graph in the sense of Definition~\ref{Def2.2}. In particular, the compact core $\mathcal K$ is nonempty and contains at least one bounded edge, and there exist $M\in\mathbb N$ and half lines $e_{1},\dots,e_{M}$ with pairwise disjoint interiors such that each $e_{j}$ is isometric to $[0,\infty)$ and
\[
\mathcal G=\mathcal K \cup \bigcup_{j=1}^{M} e_{j}, \qquad \mathcal K\cap e_{j}=\{v_{j}\},
\]
where $v_j$ is the attachment vertex of $e_j$ to the core. We use the reference point \(0\) fixed in Section~1. For $R>0$ we denote by $B_R(y)$ the metric ball in $\mathcal G$ centered at $y\in\mathcal G$ and by $B_R(0)^{c}=\mathcal G\setminus B_R(0)$.

The energy functional $J_{\lambda}:Y\to\mathbb{R}$ associated with \eqref{eq3.1} is given by
\[
J_{\lambda}(u) = \frac{1}{2}\big(\|u^{+}\|^{2}-\|u^{-}\|^{2}\big) + \frac{\lambda}{2}\int_{\mathcal{G}} |u|^{2}\,dx - \int_{\mathcal{G}} F(|u|)\,dx .
\]
A direct computation shows that $J_{\lambda}\in C^{1}(Y,\mathbb{R})$ and
\[
J_{\lambda}'(u)[v] = (u^{+},v^{+}) - (u^{-},v^{-}) + \lambda\,\operatorname{Re}\int_{\mathcal{G}} u\cdot \bar v\,dx - \operatorname{Re}\int_{\mathcal{G}} f(|u|)\,u\cdot \bar v\,dx \quad \text{for all } u,v\in Y .
\]

We define the ground level on the generalized Nehari set by
\[
d_{\lambda} = \inf_{u\in \mathcal{N}_{\lambda}} J_{\lambda}(u),
\]
where the generalized Nehari set is
\begin{equation}\label{eq3.2}
\mathcal{N}_{\lambda} = \Bigl\{\, u\in Y\setminus Y^{-}:\ J_{\lambda}'(u)[u]=0 \ \text{and}\ J_{\lambda}'(u)[v]=0\ \text{for all } v\in Y^{-}\Bigr\}.
\end{equation}
The fibres $Y(u)$ and $\hat Y(u)$ are those defined in Section~\ref{Subsec2.3}.

The compactness obstruction on a noncompact graph is the escape of mass along an external edge. If a sequence concentrates farther and farther away from the vertex on an edge \(e_j\simeq [0,\infty)\), then after recentering around its concentration point the limiting geometry is the full line \(\mathbb R\). Therefore the autonomous level at infinity is defined by the full-line problem.

Let \(J_{\lambda,\mathbb R}:Y_{\mathbb R}\to\mathbb R\) be the autonomous functional associated with the Dirac operator on \(\mathbb R\), and let \(\mathcal N_{\lambda,\mathbb R}\) be the corresponding generalized Nehari set. We set
\[
d_\lambda^\infty := d_{\lambda,\mathbb R} := \inf_{u\in\mathcal N_{\lambda,\mathbb R}}J_{\lambda,\mathbb R}(u).
\]

For \(V\equiv\lambda\), the notation in Section~\ref{Subsec2.3} becomes
\[
I=J_\lambda,\qquad \mathcal M=\mathcal N_\lambda,\qquad \ell=d_\lambda.
\]
The following consequences of the common reduction give the autonomous fibre properties and energy estimates used below.

\begin{Lemma}\label{Lem3.1}
For all $u\in\mathcal N_{\lambda}$ one has
\begin{equation}\label{eq3.3}
J_{\lambda}(t u+\xi)<J_{\lambda}(u) \quad \text{for all } t\ge 0,\ \xi\in Y^{-},\ \text{with }(t,\xi)\ne(1,0).
\end{equation}
Hence $u$ is the unique global maximum of $J_{\lambda}$ on $\hat Y(u)=Y^{-}\oplus[0,\infty)u$.
\end{Lemma}

\begin{proof}
For \(u\in\mathcal N_\lambda\), set \(w=((t^2-1)/2)u+t\xi\in Y(u)\). The Nehari identities imply \(J_\lambda'(u)[w]=0\). Expanding \(J_\lambda(tu+\xi)-J_\lambda(u)\) and using this identity gives
\[
\begin{aligned}
J_\lambda(tu+\xi)-J_\lambda(u) &=\int_{\mathcal G}\!\Big[\operatorname{Re}\Big(f(|u|)u\cdot\overline{\frac{t^2-1}{2}u+t\xi}\Big)+F(|u|)-F(|tu+\xi|)\Big]dx \\
&\quad -\frac12\|\xi\|^2+\frac\lambda2\|\xi\|_2^2 .
\end{aligned}
\]
The integral is nonpositive by Lemma~\ref{Lem2.13}. If \(\xi\ne0\), Lemma~\ref{Lem2.7} gives
\[
-\frac12\|\xi\|^2+\frac\lambda2\|\xi\|_2^2 \le -\frac12\Big(1-\frac{|\lambda|}{mc^2}\Big)\|\xi\|^2<0.
\]
If \(\xi=0\) and \(t\ne1\), Lemma~\ref{Lem2.13} is strict on the set \(\{u\ne0\}\), which has positive measure. Hence \eqref{eq3.3} follows, and the uniqueness of the fibre maximum is immediate.
\end{proof}

\begin{Lemma}\label{Lem3.2}
\begin{itemize}
\item [$(a)$]There exists $\alpha>0$ such that
\[
d_{\lambda}=\inf_{u\in\mathcal N_{\lambda}} J_{\lambda}(u) \ \ge\ \inf_{u\in\mathcal S_{\alpha}} J_{\lambda}(u)\ >\ 0,
\]
where $\mathcal S_{\alpha}=\{u\in Y^{+}:\ \|u\|=\alpha\}$.
\item [$(b)$]For every $u\in\mathcal N_{\lambda}$,
\[
\|u^{+}\|\ \ge\ \max\!\left\{ \sqrt{\frac{2 m c^{2}\, d_{\lambda}}{m c^{2}+|\lambda|}}, \ \sqrt{\frac{mc^2-|\lambda|}{mc^2+|\lambda|}}\,\|u^{-}\| \right\}.
\]
\end{itemize}
\end{Lemma}

\begin{proof}
(a) From \((f_1)\), for every small \(\varepsilon>0\) there is \(C_\varepsilon>0\) such that
\[
F(t)\le \varepsilon t^2+C_\varepsilon t^p\qquad (t\ge0).
\]
Thus, for \(u\in Y^+\), Lemma~\ref{Lem2.7} and the embedding \(Y\hookrightarrow L^p\) give
\[
J_\lambda(u)\ge \Big[\frac12-\frac{|\lambda|/2+\varepsilon}{mc^2}\Big]\|u\|^2-C_\varepsilon\|u\|^p.
\]
Choosing \(\varepsilon\) small and then \(\alpha>0\) small gives \(\inf_{\mathcal S_\alpha}J_\lambda>0\). The fibre characterization in Proposition~\ref{Pro2.14}, with \(V\equiv\lambda\), then yields
\[
d_\lambda=\inf_{w\in Y^+\setminus\{0\}}\max_{\hat Y(w)}J_\lambda\ge \inf_{\mathcal S_\alpha}J_\lambda>0.
\]
(b) If \(u\in\mathcal N_\lambda\), then \(F\ge0\) and Lemma~\ref{Lem2.7} imply
\[
d_\lambda\le J_\lambda(u) \le \Big(\frac12+\frac{|\lambda|}{2mc^2}\Big)\|u^+\|^2 -\frac12\Big(1-\frac{|\lambda|}{mc^2}\Big)\|u^-\|^2.
\]
Dropping the last term gives the first estimate in (b); since \(d_\lambda>0\), keeping the negative term gives the second estimate. This proves (b).
\end{proof}

\begin{Lemma}\label{Lem3.3}
If $E\subset Y^{+}\setminus\{0\}$ is compact, then there exists $R>0$ such that
\[
J_{\lambda}\le 0 \quad \text{on } Y(u)\setminus B_{R}(0)\ \ \text{for every } u\in E,
\]
where $B_{R}(0)=\{z\in Y:\ \|z\|\le R\}$.
\end{Lemma}

\begin{proof}
After normalizing the compact set \(E\), assume \(\|u\|=1\) for all \(u\in E\). If the assertion failed, there would be \(u_n\to u\in E\) and \(w_n=s_nu_n+\eta_n\in Y(u_n)\), \(\eta_n\in Y^-\), such that \(J_\lambda(w_n)>0\) and \(\|w_n\|\to\infty\). Put \(z_n=w_n/\|w_n\|=a_nu_n+b_n\), where \(a_n=s_n/\|w_n\|\) and \(b_n=\eta_n/\|w_n\|\). From \(J_\lambda(w_n)>0\), \(F\ge0\), and Lemma~\ref{Lem2.7}, one obtains \(|a_n|\ge c_0>0\). Hence, up to a subsequence, \(z_n\rightharpoonup z\) in \(Y\) with \(z^+\ne0\). By local compactness, \(z_n\to z\) a.e. on a set of positive measure. On this set \(|w_n|\to\infty\), and Lemma~\ref{Lem2.8} gives
\[
\int_{\mathcal G}\frac{F(|w_n|)}{\|w_n\|^2}\,dx\to\infty.
\]
Consequently \(J_\lambda(w_n)/\|w_n\|^2\to-\infty\), contradicting \(J_\lambda(w_n)>0\). The lemma follows.
\end{proof}

\begin{Lemma}\label{Lem3.4}
For each $u\in Y\setminus Y^{-}$ the set $\mathcal N_{\lambda}\cap \hat Y(u)$ consists of exactly one point $\hat m(u)$, which is the unique global maximum of $J_{\lambda}$ on $\hat Y(u)$.
\end{Lemma}

\begin{proof}
Apply Proposition~\ref{Pro2.14} with \(V\equiv\lambda\). Its strict fibre inequality shows that every element of \(\mathcal N_\lambda\cap\hat Y(u)\) is the unique maximizer on that fibre. We denote it by \(\hat m(u)\).
\end{proof}

\begin{Lemma}\label{Lem3.5}
$J_{\lambda}$ is coercive on $\mathcal N_{\lambda}$, that is,
\[
J_{\lambda}(u)\to\infty \qquad \text{as } \|u\|\to\infty,\ u\in\mathcal N_{\lambda}.
\]
\end{Lemma}

\begin{proof}
Assume that \(u_n\in\mathcal N_\lambda\), \(\|u_n\|\to\infty\), and \(J_\lambda(u_n)\le d\). Put \(v_n=u_n/\|u_n\|\). Lemma~\ref{Lem3.2}(b) gives \(\|v_n^+\|\ge\xi>0\). Since \(J_\lambda'(u_n)[u_n^+]=0\),
\[
\|v_n^+\|^2+\lambda\|v_n^+\|_2^2 =\frac1{\|u_n\|^2}\operatorname{Re}\int_{\mathcal G}f(|u_n|)u_n\cdot\overline{u_n^+}\,dx.
\]
For fixed \(\tau>0\), decompose \(f(s)s=g_\tau(s)+j_\tau(s)\) as in the proof of Lemma~\ref{Lem2.11}, with \(r=p/(p-1)\). The small part is bounded by \(\tau/(mc^2)\). For the large part, H\"older's inequality, the same decomposition, and \((f_2)\) give
\[
\frac1{\|u_n\|^2}\int |j_\tau(|u_n|)|\,|u_n^+|\,dx \le \frac{C}{\|u_n\|}\Big(\int_{\mathcal G}\big(\tfrac12 f(|u_n|)|u_n|^2-F(|u_n|)\big)dx\Big)^{1/r}=o(1),
\]
because the last integral equals \(J_\lambda(u_n)\le d\). Hence
\[
\Big(1-\frac{|\lambda|}{mc^2}\Big)\|v_n^+\|^2\le \frac{\tau}{mc^2}+o(1).
\]
Letting \(n\to\infty\) and then \(\tau\downarrow0\) contradicts \(\|v_n^+\|\ge\xi\). Therefore \(J_\lambda\) is coercive on \(\mathcal N_\lambda\).
\end{proof}

\begin{Lemma}\label{Lem3.6}
The map $\hat m: Y^{+}\setminus\{0\}\to \mathcal N_{\lambda}$ is continuous.
\end{Lemma}

\begin{proof}
It is enough to consider \(u_n\to u\) in \(S^+\). By Lemmas~\ref{Lem3.3} and \ref{Lem3.5}, the sequence \(\hat m(u_n)=t_nu_n+w_n^-\) is bounded; moreover \(t_n\ge c_0>0\). Passing to a subsequence,
\[
t_n\to t>0, \qquad w_n^-\rightharpoonup w^- \quad\text{in }Y^- .
\]
For every \(s\ge0\), \(v\in Y^-\), the maximality of \(\hat m(u_n)\) gives
\[
J_\lambda(\hat m(u_n))\ge J_\lambda(su_n+v)\to J_\lambda(su+v).
\]
On the other hand, the weak upper semicontinuity argument in Proposition~\ref{Pro2.14} gives
\[
J_\lambda(tu+w^-)\ge\limsup_n J_\lambda(\hat m(u_n)).
\]
Hence \(tu+w^-\) is a global maximizer of \(J_\lambda\) on \(\hat Y(u)\), and by uniqueness it equals \(\hat m(u)\). Thus \(\hat m(u_n)\rightharpoonup\hat m(u)\), and the corresponding energies converge. Written with the equivalent negative norm \(\|w\|_\diamond^2=\|w\|^2-\lambda\|w\|_2^2\), this energy convergence and the weak lower semicontinuity of \(\|\cdot\|_\diamond^2\) and \(\int F\) imply \(\|w_n^-\|_\diamond\to\|w^-\|_\diamond\). Hence \(w_n^-\to w^-\) in \(Y^-\), and \(\hat m(u_n)\to\hat m(u)\) in \(Y\).
\end{proof}

Define
\[
\hat\Psi_\lambda:Y^+\setminus\{0\}\to\mathbb R,\qquad \hat\Psi_\lambda(w)=J_\lambda(\hat m(w)).
\]
\begin{Proposition}\label{Pro3.7}
\(\hat{\Psi}_{\lambda}\in C^{1}\big(Y^{+}\setminus\{0\},\mathbb{R}\big)\) and
\[
\hat{\Psi}_{\lambda}'(w)\,z = \frac{\|\hat m(w)^{+}\|}{\|w\|}\, J_{\lambda}'\big(\hat m(w)\big)\,z, \qquad \text{for } w,z\in Y^{+},\ w\neq 0.
\]
\end{Proposition}

\begin{proof}
Let \(u=\hat m(w)=u^-+s_0w\), where \(s_0=\|u^+\|/\|w\|\). By Lemma~\ref{Lem3.6}, \(\hat m(w+tz)\to u\) as \(t\to0\). Comparing the maximal values on the two fibres \(\hat Y(w)\) and \(\hat Y(w+tz)\), and applying the mean value theorem along the segment in the positive direction, gives, as in Lemma~\ref{Lem2.16},
\[
\frac{\hat\Psi_\lambda(w+tz)-\hat\Psi_\lambda(w)}{t}\to s_0J_\lambda'(u)[z] =\frac{\|\hat m(w)^+\|}{\|w\|}J_\lambda'(\hat m(w))[z].
\]
The right-hand side depends continuously on \(w\), since \(\hat m\) and \(J_\lambda'\) are continuous. Thus \(\hat\Psi_\lambda\in C^1\) and the displayed formula follows.
\end{proof}

Proposition~\ref{Pro2.14} and Lemma~\ref{Lem3.4} show that \(\hat m|_{S^+}:S^+\to\mathcal N_\lambda\) is a homeomorphism, with inverse
\[
\check m(u)=\frac{u^+}{\|u^+\|}.
\]
Let \(\Psi_\lambda=\hat\Psi_\lambda|_{S^+}\).

\begin{Corollary}\label{Cor3.8}
\begin{itemize}
\item[(a)] \(\Psi_{\lambda}\in C^{1}(S^{+})\) and
\[
\Psi_{\lambda}'(w)\,z = \|\hat m(w)^{+}\|\, J_{\lambda}'(\hat m(w))\,z, \qquad \text{for } z\in T_{w}S^{+}=\{\,v\in Y^{+}:\ (w,v)=0\,\}.
\]
\item[(b)] A sequence \((w_{n})\subset S^{+}\) is a Palais--Smale sequence for \(\Psi_{\lambda}\) if and only if \((\hat m(w_{n}))\subset \mathcal N_{\lambda}\) is a Palais--Smale sequence for \(J_{\lambda}\).
\item[(c)] We have
\[
\inf_{S^{+}}\Psi_{\lambda} \;=\; \inf_{\mathcal N_{\lambda}} J_{\lambda} \;=\; d_{\lambda}.
\]
Moreover, \(w\in S^{+}\) is a critical point of \(\Psi_{\lambda}\) if and only if \(\hat m(w)\in\mathcal N_{\lambda}\) is a critical point of \(J_{\lambda}\), and the corresponding critical values coincide.
\end{itemize}
\end{Corollary}

\begin{proof}
The derivative formula follows from Proposition~\ref{Pro3.7}. The Palais--Smale assertions follow from Corollary~\ref{Cor2.17} with \(V\equiv\lambda\). The equality of the infima follows from the definition \eqref{eq3.2} and the bijectivity of \(\hat m|_{S^+}\).
\end{proof}

The autonomous level also gives a lower bound for the nonautonomous ground level.

\begin{Lemma}\label{Lem3.9}
Assume \((V_1)\). Then
\[
\ell\ge d_{V_0}>0.
\]
\end{Lemma}

\begin{proof}
Since \(V(x)\ge V_0\) on \(\mathcal G\), for every \(u\in Y\),
\[
I(u) = \frac{1}{2}\big(\|u^{+}\|^{2}-\|u^{-}\|^{2}\big) +\frac{1}{2}\int_{\mathcal G}V(x)|u|^{2}\,dx -\int_{\mathcal G}F(|u|)\,dx \ge J_{V_0}(u).
\]
Taking the maximum over \(\hat Y(u)\), and then the infimum over \(u\in Y^+\setminus\{0\}\), yields
\[
\ell\ge d_{V_0}.
\]
The positivity \(d_{V_0}>0\) follows from Lemma~\ref{Lem3.2}(a).
\end{proof}

\subsection{Compactness and ground states}\label{Subsec3.2}

Throughout this subsection, assume \((f_1)\)--\((f_2)\), \(V\in L^\infty(\mathcal G,\mathbb R)\), \(\|V\|_\infty<mc^2\), and \(V(x)\to V_\infty\) uniformly as \(d(x,0)\to\infty\). The full-line level \(d_{V_\infty}^{\infty}\) bounds the energy carried by a nonzero profile escaping along an external edge. The order assumptions in \((V_1)\) are used only in the stated comparisons of ground levels.

For a bounded sequence \((u_n)\subset Y\), define
\[
\mathcal Q_n(R)=\sup_{y\in\mathcal G}\int_{B_R(y)}|u_n|^2\,dx,\qquad R>0.
\]
\begin{Lemma}\label{Lem3.10}
Let \((u_n)\subset Y\) be bounded and assume that there exist \(R_0>0\) and \(\eta_0>0\) such that \(\mathcal Q_n(R_0)\ge \eta_0\) for all \(n\). Then there exist points \(y_n\in\mathcal G\) and constants \(R>0\), \(\eta>0\) such that
\[
\int_{B_R(y_n)}|u_n|^2\,dx \ge \eta \quad \text{for all }n.
\]
Moreover, up to a subsequence, exactly one of the following alternatives holds:
\begin{enumerate}
\item[(i)] \(\sup_n \mathrm{dist}(y_n,\mathcal K)<\infty\);
\item[(ii)] \(\mathrm{dist}(y_n,\mathcal K)\to\infty\), and there exist an external edge \(e_j\) and numbers \(t_n\to\infty\) such that \(y_n\in e_j\) and \(\mathrm{dist}(y_n,v_j)=t_n\) for all \(n\).
\end{enumerate}
\end{Lemma}

\begin{proof}
Choose \(y_n\) so that \(\int_{B_{R_0}(y_n)}|u_n|^2dx\ge\eta_0/2\). If \(\operatorname{dist}(y_n,\mathcal K)\) is bounded, alternative (i) holds. Otherwise, pass to a subsequence with \(\operatorname{dist}(y_n,\mathcal K)\to\infty\). For large \(n\), \(y_n\) lies on one of the finitely many external edges. Passing to a subsequence fixes this edge, say \(e_j\), and then \(t_n=\operatorname{dist}(y_n,v_j)=\operatorname{dist}(y_n,\mathcal K)\to\infty\). This is alternative (ii).
\end{proof}

\begin{Definition}\label{Def3.11}
Assume that \(y_n\in e_j\) and
\[
t_n:=\mathrm{dist}(y_n,v_j)\to\infty .
\]
We identify \(e_j\) with \([0,\infty)\) by the arc-length coordinate \(x=\mathrm{dist}(\cdot,v_j)\). Let \(\zeta\in C^\infty([0,\infty))\) be fixed with
\[
0\le \zeta\le 1,\qquad \zeta(x)=0\ \text{for }0\le x\le 1,\qquad \zeta(x)=1\ \text{for }x\ge 2 .
\]
For \(u_n\in Y\), define the recentered tail on the full line by
\[
v_n(s)=
\begin{cases}
\zeta(t_n+s)\,u_n|_{e_j}(t_n+s), & t_n+s\ge 0,\\[2mm]
0, & t_n+s<0,
\end{cases}
\qquad s\in\mathbb R .
\]
We say that \((u_n)\) admits a nontrivial profile at infinity along \(e_j\) if there exists \(u^\infty\in Y_{\mathbb R}\setminus\{0\}\) such that, up to a subsequence,
\[
v_n\rightharpoonup u^\infty \quad\text{in }Y_{\mathbb R}, \qquad v_n\to u^\infty \quad\text{in }L^2_{\mathrm{loc}}(\mathbb R,\mathbb C^2).
\]
\end{Definition}

The next three lemmas give the nonlinear tail estimates at the autonomous ground level.

\begin{Lemma}\label{Lem3.12}
Let \(p>2\) and set \(r=\frac{p}{p-1}\in(1,2)\). Assume \(f:[0,\infty)\to\mathbb R\) is continuous with \(f(0)=0\), and there exists \(c_1>0\) such that
\[
0\le f(s)\le c_1\bigl(1+s^{p-2}\bigr)\qquad\text{for all } s\ge 0 .
\]
Then for each \(\tau>0\) there exist \(\delta=\delta(\tau)\in(0,1]\) and \(c_\tau>0\) such that, with
\[
\chi_\delta(s)=\mathbf 1_{(0,\delta)}(s),\qquad g_\tau(s)=\chi_\delta(s)\,f(s)\,s,\qquad j_\tau(s)=(1-\chi_\delta(s))\,f(s)\,s,
\]
the following hold for all \(s\ge 0\):
\[
|g_\tau(s)|\le \tau\,s, \qquad |j_\tau(s)|^{\,r}\le c_\tau\, s^{2}\,f(s).
\]
The same statements hold for \(t\in\mathbb R\) upon writing \(s=|t|\) and \(f(|t|)\).
\end{Lemma}

\begin{proof}
Choose \(\delta\in(0,1]\) so that \(f(s)\le\tau\) on \([0,\delta]\). Then \(|g_\tau(s)|\le\tau s\). On \([\delta,\infty)\),
\[
|j_\tau(s)|^r=(f(s))^r s^r=\big(f(s)^{r-1}s^{r-2}\big)f(s)s^2.
\]
The factor \(f(s)^{r-1}s^{r-2}\) is bounded on \([\delta,1]\) by continuity, and on \([1,\infty)\) by \((f_1)\), since \((p-2)(r-1)+r-2=0\). This proves the second estimate; the extension to real arguments is immediate.
\end{proof}

\begin{Lemma}\label{Lem3.13}
Let \((u_{n})\subset Y\) be a \((PS)_{d_{\lambda}}\) sequence for \(J_{\lambda}\) with \(u_{n}\rightharpoonup u\) in \(Y\). Assume \((f_1)\)--\((f_2)\) and, in addition, that \(u\neq 0\) and \(J_{\lambda}'(u)=0\). Then for every \(\tau>0\) there exist \(R>0\) and \(n_{0}\in\mathbb N\) such that
\[
\int_{B_{R}(0)^{c}} f\!\left(|u_{n}|\right)\,|u_{n}|\,|u_{n}^{+}|\,dx \;\le\; 2\,\Lambda\,\tau \quad\text{for all } n\ge n_{0},
\]
where
\[
\Lambda=\max\Bigg\{\,\sup_{n\in\mathbb N}\Big(\int_{\mathcal G}|u_{n}^{+}|^{p}\,dx\Big)^{\!1/p}\,,\; \sup_{n\in\mathbb N}\int_{\mathcal G}|u_{n}|\,|u_{n}^{+}|\,dx\Bigg\}.
\]
\end{Lemma}

\begin{proof}
Set \(r=p/(p-1)\) and split \(f(s)s=g_\tau(s)+j_\tau(s)\) as in Lemma~\ref{Lem3.12}. H\"older's inequality gives, for every \(R>0\),
\[
\int_{B_R(0)^c} f(|u_n|)|u_n|\,|u_n^+|dx \le \tau\Lambda+C_\tau\Lambda\Big(\int_{B_R(0)^c} f(|u_n|)|u_n|^2dx\Big)^{1/r}.
\]
By \((f_2)\), \(f(t)t^2\le C\big(\frac12 f(t)t^2-F(t)\big)\). Let
\[
g_n=\frac12 f(|u_n|)|u_n|^2-F(|u_n|),\qquad g=\frac12 f(|u|)|u|^2-F(|u|).
\]
Since \((u_n)\) is a \((PS)_{d_\lambda}\) sequence, \(\int g_n\to d_\lambda\). Every subsequence has a further subsequence converging to \(u\) almost everywhere, by local compactness and a diagonal argument. Along such a subsequence, \(g_n\to g\) almost everywhere. The assumptions \(u\ne0\) and \(J_\lambda'(u)=0\) imply \(u\in\mathcal N_\lambda\), so Fatou's lemma yields \(J_\lambda(u)=d_\lambda\). Scheff\'e's lemma gives \(g_n\to g\) in \(L^1(\mathcal G)\) along this further subsequence. The subsequence criterion therefore gives the same convergence for the original sequence. Choose \(R\) so that \(\int_{B_R(0)^c}g\) is small and then \(n\) large. The displayed estimate gives the asserted tail bound.
\end{proof}

\begin{Lemma}\label{Lem3.14}
Let \((u_{n})\subset Y\) be a \((PS)_{d_{\lambda}}\) sequence for \(J_{\lambda}\) with \(u_{n}\rightharpoonup u\) in \(Y\). Assume \(u\neq 0\) and \(J_{\lambda}'(u)=0\). Set
\[
Q_{n}=f(|u_{n}|)\,u_{n}\cdot\overline{u_{n}^{+}}-f(|u|)\,u\cdot\overline{u^{+}}.
\]
Then
\[
\int_{\mathcal G} |Q_{n}|\,dx \to 0 \quad \text{as } n\to\infty .
\]
The same conclusion holds with the negative spectral parts in place of the positive parts.
\end{Lemma}

\begin{proof}
Fix \(\tau>0\). Choose \(R>0\) so that \(\int_{B_R(0)^c}f(|u|)|u|\,|u^+|dx<\tau\), and enlarge \(R\) if necessary so that the uniform tail estimate in Lemma~\ref{Lem3.13} also holds. On the compact subgraph \(B_R(0)\), the embedding \(Y\hookrightarrow L^s(B_R(0))\) is compact for all finite \(s\), and the Nemytskii map \(z\mapsto f(|z|)z\) is continuous from \(L^p\) to \(L^{p/(p-1)}\). Hence
\[
\int_{B_R(0)}|Q_n|dx\to0.
\]
On the tail, Lemma~\ref{Lem3.13} gives
\[
\int_{B_R(0)^c} f(|u_n|)|u_n|\,|u_n^+|dx\le2\Lambda\tau
\]
for all large \(n\), while the corresponding integral for \(u\) is below \(\tau\). Thus \(\limsup_n\int_{\mathcal G}|Q_n|dx\le(2\Lambda+1)\tau\), and the arbitrariness of \(\tau\) proves the claim. The proof with the negative spectral parts is identical, since their \(L^2\) and \(L^p\) norms satisfy the same bounds.
\end{proof}

We now establish the corresponding compactness tools for the variable-potential functional.

\begin{Lemma}\label{Lem3.15}
Let \((v_n)\subset \mathcal M\) be a \((PS)_\beta\) sequence for \(I\) with \(\beta>0\). Then for every fixed \(R>0\) there exist \(\eta>0\) and points \(y_n\in\mathcal G\) such that
\[
\int_{B_R(y_n)} |v_n(x)|^{2}\,dx \ge \eta \quad\text{for all }n.
\]
\end{Lemma}

\begin{proof}
If the assertion failed for some \(R>0\), a subsequence of \((v_n)\) would vanish in the sense of Lions. Since \((v_n)\subset\mathcal M\) and \(I(v_n)\to \beta\), Proposition~\ref{Pro2.15} gives boundedness in \(Y\); hence Lemma~\ref{Lem2.10} implies \(v_n\to0\) in \(L^q\) for every \(q>2\). Using \((f_1)\) and the usual small-large splitting, one gets
\[
\int_{\mathcal G}F(|v_n|)dx\to0, \qquad \int_{\mathcal G}f(|v_n|)|v_n|^2dx\to0.
\]
Since \(I'(v_n)[v_n]=o(1)\), it follows that \(I(v_n)\to0\), contradicting \(\beta>0\).
\end{proof}

\begin{Lemma}\label{Lem3.16}
Let \(u_n\rightharpoonup u\) in \(Y\), assume that \(u_n(x)\to u(x)\) for a.e.\ \(x\in\mathcal G\), and set
\[
\omega_n=u_n-u.
\]
Then:
\begin{enumerate}
\item[\textnormal{(i)}]
\[
\int_{\mathcal G} \Big(F(|u_n|)-F(|u|)-F(|\omega_n|)\Big)\,dx=o_n(1);
\]
\item[\textnormal{(ii)}]
\[
f(|u_n|)u_n-f(|u|)u-f(|\omega_n|)\omega_n \to0 \quad \text{in } L^2(\mathcal G,\mathbb C^2) + L^{\frac{p}{p-1}}(\mathcal G,\mathbb C^2),
\]
where \(p>2\) is given by \((f_1)\).
\end{enumerate}
In particular, for every \(v\in Y\),
\[
\operatorname{Re}\int_{\mathcal G} \big(f(|u_n|)u_n-f(|u|)u-f(|\omega_n|)\omega_n\big) \cdot\overline v\,dx =o_n(1).
\]
\end{Lemma}

\begin{proof}
Put $g(z)=f(|z|)z$ and choose a finite $q>2p$. The sequences $u_n$, $\omega_n$ and $\omega_n+tu$, $0\le t\le1$, are uniformly bounded in $L^2\cap L^q$. Let $\rho\in C_c^\infty([0,\infty))$ equal one on $[0,1]$ and vanish on $[2,\infty)$, and set $g_R(z)=\rho(|z|/R)g(z)$. Since $f\in C^1([0,\infty))$, $g_R$ is globally Lipschitz and $g_R(0)=0$. For fixed $R$,
\[
|g_R(u_n)-g_R(\omega_n)-g_R(u)| \le2\operatorname{Lip}(g_R)|u|.
\]
The left-hand side converges to zero almost everywhere. Dominated convergence gives convergence to zero in $L^2$.

For $R\ge1$, the growth assumption yields
\[
\|g(z)-g_R(z)\|_2^2 \le C\int_{\{|z|>R\}}|z|^{2p-2}\,dx \le C R^{2p-2-q}\|z\|_q^q.
\]
This estimate is uniform for the sequences in question. First pass to the limit in $n$ and then let $R\to\infty$. We obtain the stronger conclusion
\begin{equation}\label{eq3.4}
g(u_n)-g(\omega_n)-g(u)\longrightarrow0\quad\text{in }L^2.
\end{equation}
In particular, (ii) follows.

For (i), apply the same argument to $\omega_n+tu$ for each fixed $t\in[0,1]$. This gives
\[
g(\omega_n+tu)-g(\omega_n)-g(tu)\longrightarrow0\quad\text{in }L^2.
\]
All these expressions are uniformly bounded in $L^2$, independently of $t$ and $n$. Moreover $g(\omega_n)\rightharpoonup0$ in $L^2$, because it is bounded there and converges to zero almost everywhere. Using the fundamental theorem of calculus in $\mathbb C^2\cong\mathbb R^4$, we have
\begin{align*}
&\int_{\mathcal G}\bigl(F(|\omega_n+u|)-F(|\omega_n|)-F(|u|)\bigr)\,dx\\
&=\operatorname{Re}\int_{\mathcal G}g(\omega_n)\cdot\overline u\,dx\\
&\quad+\int_0^1\operatorname{Re}\int_{\mathcal G} [g(\omega_n+tu)-g(\omega_n)-g(tu)]\cdot\overline u\,dx\,dt\longrightarrow0.
\end{align*}
The last limit follows from dominated convergence in $t$. Finally $Y\hookrightarrow L^2$ turns \eqref{eq3.4} into convergence in $Y^*$.
\end{proof}

\begin{Lemma}\label{Lem3.17}
Let \((u_n)\subset Y\) be a \((PS)_\beta\) sequence for \(I\) such that
\[
u_n\rightharpoonup u \ \text{in }Y, \qquad u_n(x)\to u(x)\ \text{for a.e. }x\in\mathcal G, \qquad I'(u)=0.
\]
Set \(\omega_n=u_n-u\). Then:
\begin{enumerate}
\item[\textnormal{(a)}] \(I(u_n)=I(u)+I(\omega_n)+o_n(1)\);
\item[\textnormal{(b)}] \(I'(\omega_n)\to0\) in \(Y^*\).
\end{enumerate}
\end{Lemma}

\begin{proof}
The quadratic part splits by weak convergence in the Hilbert spaces \(Y^\pm\) and in \(L^2\) with the bounded multiplier \(V\). Lemma~\ref{Lem3.16}(i) gives the nonlinear splitting, hence (a). For every \(v\in Y\), subtracting \(I'(u)[v]=0\) from \(I'(u_n)[v]=o(1)\) and using Lemma~\ref{Lem3.16}(ii) gives \(I'(\omega_n)[v]=o(1)\) uniformly for \(\|v\|\le1\), which proves (b).
\end{proof}

\begin{Lemma}\label{Lem3.18}
Define \(A_V:Y\to Y\) by
\[
(A_Vu,v) = (u^+,v^+)-(u^-,v^-) + \operatorname{Re}\int_{\mathcal G}V(x)u\cdot\overline v\,dx \qquad \forall\,u,v\in Y.
\]
Then \(A_V\) is a bounded isomorphism of \(Y\).
\end{Lemma}

\begin{proof}
Let \(\mathcal J=P^+-P^-\) and define \(K_V\) by \((K_Vu,v)=\operatorname{Re}\int Vu\cdot\overline v\). Lemma~\ref{Lem2.7} gives \(\|K_V\|\le \|V\|_\infty/(mc^2)<1\). Since \(A_V=\mathcal J+K_V=\mathcal J(\operatorname{Id}+\mathcal JK_V)\), the operator \(\operatorname{Id}+\mathcal JK_V\) is invertible by the Neumann series, and so is \(A_V\).
\end{proof}

\begin{Proposition}\label{Pro3.19}
Assume $(f_1)$--$(f_2)$, $V\in L^\infty(\mathcal G,\mathbb R)$, $\|V\|_\infty<mc^2$, and $V(x)\to V_\infty$ uniformly as $d(x,0)\to\infty$. Let $(z_n)$ be bounded in $Y$ and satisfy
\[
z_n\rightharpoonup0,\qquad I'(z_n)\longrightarrow0\text{ in }Y^*, \qquad I(z_n)\longrightarrow\beta.
\]
If $z_n$ does not converge to zero in $Y$, then, after passing to a subsequence, there are an external edge $e_j$, numbers $t_n\to\infty$, and a nonzero $z^\infty\in Y_{\mathbb R}$ such that the recentered cutoffs converge weakly to $z^\infty$ in $Y_{\mathbb R}$ and strongly in every local finite $L^q$ space. Moreover,
\[
J'_{V_\infty,\mathbb R}(z^\infty)=0, \qquad \beta\ge J_{V_\infty,\mathbb R}(z^\infty)\ge d_{V_\infty}^{\infty}.
\]
In particular, $\beta<d_{V_\infty}^{\infty}$ implies $z_n\to0$ in $Y$.
\end{Proposition}

\begin{proof}
If the sequence vanishes in the sense of Lemma~\ref{Lem2.10}, then $z_n\to0$ in $L^q$ for every $q>2$. For every $\tau>0$, the bound
\[
|f(s)s|\le\tau s+C_\tau s^{p-1}
\]
and the embeddings of $Y$ give $f(|z_n|)z_n\to0$ in $Y^*$. The invertibility in Lemma~\ref{Lem3.18}, applied to the identity $I'(z_n)=A_Vz_n-g(z_n)$ after Riesz identification, then gives $z_n\to0$ in $Y$.

Otherwise there exist fixed $R,\eta>0$ and $y_n\in\mathcal G$ such that $\int_{B_R(y_n)}|z_n|^2\ge\eta$. Apply Lemma~\ref{Lem3.10}. Local compactness and the weak limit zero exclude centers at bounded distance from the core. Thus, after a subsequence, the centers lie on one external edge $e_j\simeq[0,\infty)$ at coordinates $t_n\to\infty$.

Let $\zeta$ be the cutoff from Definition~\ref{Def3.11}. Multiplication by $\zeta$ on $e_j$, followed by extension by zero to $\mathbb R$, is bounded from $L^2(\mathcal G)$ to $L^2(\mathbb R)$ and from $\operatorname{dom}\mathcal D$ to $H^1(\mathbb R)$. The latter assertion follows because $\zeta$ vanishes in a fixed neighborhood of the attachment vertex. Interpolation, as in Lemma~\ref{Lem2.12}, therefore gives a bounded map $Y\to Y_{\mathbb R}$. Translation on $\mathbb R$ is an isometry for the full-line form norm. The recentered sequence $v_n$ is consequently bounded in $Y_{\mathbb R}$. Local compactness gives a weak limit $z^\infty$ and strong local $L^q$ convergence for every finite $q\ge2$. The mass bound shows that $z^\infty\ne0$.

For each fixed $A>0$, the assumed limit of $V$ implies
\begin{equation}\label{eq3.5}
\sup_{|s|\le A}\left|(V|_{e_j})(t_n+s)-V_\infty\right|\longrightarrow0.
\end{equation}
Indeed, the whole translated interval lies outside every fixed metric ball for all large $n$. For $\varphi\in C_c^\infty(\mathbb R,\mathbb C^2)$, place its translate on $e_j$ and extend by zero to the other edges. These test functions have uniformly bounded $Y$ norms, as follows by interpolation from their fixed $L^2$ and operator-domain norms. On their supports $\zeta=1$ for large $n$. Testing $I'(z_n)$ and using \eqref{eq3.5} gives
\[
J'_{V_\infty,\mathbb R}(v_n)[\varphi]\longrightarrow0.
\]
The local convergence of $v_n$ passes this identity to the limit. Density gives $J'_{V_\infty,\mathbb R}(z^\infty)=0$ on $Y_{\mathbb R}$.

The function $H(s)=\tfrac12f(s)s^2-F(s)$ is nonnegative and nondecreasing, since $H'(s)=\tfrac12f'(s)s^2\ge0$. Boundedness and the Palais--Smale identity yield $\int H(|z_n|)\to\beta$. Since $0\le\zeta\le1$, Fatou's lemma gives
\[
J_{V_\infty,\mathbb R}(z^\infty) =\int_{\mathbb R}H(|z^\infty|)\,ds \le\liminf_n\int_{\mathbb R}H(|v_n|)\,ds\le\beta.
\]
The nonzero critical point cannot lie in the negative spectral subspace, whose quadratic form for $\mathcal D_{\mathbb R}+V_\infty$ is strictly negative. It therefore belongs to the full-line generalized Nehari set, proving the lower bound by $d_{V_\infty}^{\infty}$.
\end{proof}

\begin{Lemma}\label{Lem3.20}
The functional $I$ satisfies the $(PS)_\beta$ condition for every $0<\beta<d_{V_\infty}^{\infty}$. The same is true of $\Upsilon$.
\end{Lemma}

\begin{proof}
Let $(u_n)$ be a $(PS)_\beta$ sequence for $I$, where $0<\beta<d_{V_\infty}^{\infty}$. Lemma~\ref{Lem2.11} gives boundedness. Passing to a subsequence, let $u_n\rightharpoonup u$ in $Y$ and $u_n\to u$ almost everywhere and in every local finite $L^q$ space. The nonlinear terms are bounded in $L^2$ and converge almost everywhere, so they converge weakly in $L^2$. Passing to the limit in $I'(u_n)[v]$ for arbitrary $v\in Y\subset L^2$ gives $I'(u)=0$.

If $u=0$, Proposition~\ref{Pro3.19} gives $u_n\to0$ in $Y$, contrary to $\beta>0$. Hence $u\ne0$. Its critical energy is $I(u)=\int H(|u|)>0$, and Fatou's lemma gives $I(u)\le \beta$. Set $\omega_n=u_n-u$. Lemma~\ref{Lem3.17} gives
\[
\omega_n\rightharpoonup0,\qquad I'(\omega_n)\to0, \qquad I(\omega_n)\to \beta-I(u),
\]
where $0\le \beta-I(u)<\beta<d_{V_\infty}^{\infty}$. Another application of Proposition~\ref{Pro3.19} gives $\omega_n\to0$ in $Y$. Thus $I$ satisfies $(PS)_\beta$. Corollary~\ref{Cor2.17} transfers this compactness to $\Upsilon$.
\end{proof}

For the autonomous functional, these results give the following alternative at and below the ground level.

\begin{Proposition}\label{Pro3.21}
Let \(0<\beta\le d_{\lambda}\), and let \((u_n)\subset Y\) be a Palais--Smale sequence for \(J_\lambda\) at level \(\beta\), that is,
\[
J_\lambda(u_n)\to \beta, \qquad J_\lambda'(u_n)\to 0 \ \text{in }Y^* .
\]
Then, up to a subsequence, exactly one of the following alternatives holds:
\begin{enumerate}
\item[(a)] \(u_n\to u\) strongly in \(Y\), where \(u\) is a nontrivial critical point of \(J_\lambda\). In this case \(\beta=d_\lambda\) and \(J_\lambda(u)=d_\lambda\);
\item[(b)] \(u_n\rightharpoonup0\) in \(Y\), and \((u_n)\) admits a nontrivial full-line profile at infinity along some external edge in the sense of Definition~\ref{Def3.11}. In this case there exists a nonzero critical point \(u^\infty\in Y_{\mathbb R}\) of \(J_{\lambda,\mathbb R}\) such that
\[
\beta\ge J_{\lambda,\mathbb R}(u^\infty)\ge d_{\lambda,\mathbb R} =d_\lambda^\infty .
\]
\end{enumerate}
In particular, if \(\beta<d_\lambda^\infty\), then alternative {\rm (b)} cannot occur.
\end{Proposition}

\begin{proof}
The Palais--Smale sequence is bounded by Lemma~\ref{Lem2.11}. Passing to a subsequence,
\[
u_n\rightharpoonup u\text{ in }Y, \qquad u_n\to u\text{ in }L^q_{\rm loc}(\mathcal G)\quad(2\le q<\infty),
\]
The nonlinear terms are bounded in $L^2$ and converge almost everywhere, hence weakly in $L^2$. Passing to the limit in the derivative against every element of $Y$ gives $J_\lambda'(u)=0$.

Assume first that \(u\ne0\). The standard sign argument excludes \(u\in Y^-\), hence \(u\in\mathcal N_\lambda\). With
\[
H(t)=\frac12f(t)t^2-F(t)\ge0,
\]
we have \(\int H(|u_n|)dx\to \beta\) and, by Fatou,
\[
d_\lambda\le J_\lambda(u)=\int H(|u|)dx\le \beta\le d_\lambda.
\]
Thus \(\beta=d_\lambda\) and \(J_\lambda(u)=d_\lambda\). Lemma~\ref{Lem3.14}, applied to the positive and negative spectral parts, gives convergence of the equivalent Hilbert norms
\[
\|w\|^2+\lambda\|w\|_2^2\quad\text{on }Y^+, \qquad \|w\|^2-\lambda\|w\|_2^2\quad\text{on }Y^-.
\]
Together with the weak convergence this yields \(u_n\to u\) in \(Y\). Hence (a) holds.

Now assume \(u=0\). Apply Proposition~\ref{Pro3.19} with \(V\equiv\lambda\). Strong convergence to zero would imply \(J_\lambda(u_n)\to0\), contrary to \(\beta>0\). Hence the proposition gives a nonzero full-line critical point \(u^\infty\) and
\[
\beta\ge J_{\lambda,\mathbb R}(u^\infty)\ge d_\lambda^\infty.
\]
This is alternative (b), which is excluded when \(\beta<d_\lambda^\infty\).
\end{proof}

\begin{Theorem}\label{Thm3.22}
Let \(\mathcal G\) be a generalized star graph and let \(\lambda\in(-mc^{2},mc^{2})\). Assume that \((f_1)\)--\((f_2)\) hold and that
\[
d_{\lambda}<d_{\lambda}^{\infty}=d_{\lambda,\mathbb R}.
\]
Then \(d_{\lambda}>0\) and the infimum \(d_{\lambda}\) is achieved by some \(u_{\lambda}\in \mathcal{N}_{\lambda}\), namely
\[
J_{\lambda}(u_{\lambda}) = d_{\lambda}.
\]
Moreover, every such minimizer \(u_{\lambda}\) is a critical point of \(J_{\lambda}\) in \(Y\). In particular, \(u_{\lambda}\neq 0\) and \(u_{\lambda}\) is a weak solution of \eqref{eq3.1}.
\end{Theorem}

\begin{proof}
By Lemma~\ref{Lem3.2}(a) we have \(d_{\lambda}>0\). By Corollary~\ref{Cor3.8}(c),
\[
\inf_{S^{+}}\Psi_{\lambda} =\inf_{\mathcal N_{\lambda}}J_{\lambda} =d_{\lambda}.
\]

Since \(\Psi_{\lambda}\in C^{1}(S^{+})\) by Corollary~\ref{Cor3.8}(a), Ekeland's variational principle yields a sequence \((w_{n})\subset S^{+}\) such that
\[
\Psi_{\lambda}(w_{n})\to d_{\lambda}, \qquad \Psi_{\lambda}'(w_{n})\to 0 \ \text{in } (T_{w_{n}}S^{+})^{*}.
\]
Define \(u_{n}=\hat m(w_{n})\in \mathcal N_{\lambda}\). Then Corollary~\ref{Cor3.8}(b)--(c) implies
\[
J_{\lambda}(u_{n})=\Psi_{\lambda}(w_{n})\to d_{\lambda}, \qquad J_{\lambda}'(u_{n})\to 0 \ \text{in } Y^{*},
\]
so \((u_{n})\) is a \((PS)_{d_{\lambda}}\) sequence for \(J_{\lambda}\).

By the strict gap assumption at \(\lambda\),
\[
d_{\lambda}<d_{\lambda}^{\infty}=d_{\lambda,\mathbb R}.
\]
Applying Proposition~\ref{Pro3.21} with \(\beta=d_\lambda\), alternative {\rm (b)} is excluded. Hence, up to a subsequence,
\[
u_{n}\to u_{\lambda}\quad\text{strongly in }Y, \qquad J_{\lambda}'(u_{\lambda})=0, \qquad J_{\lambda}(u_{\lambda})=d_{\lambda}.
\]
In particular \(u_{\lambda}\neq 0\) because \(d_{\lambda}>0\).

Moreover \(u_{\lambda}\in\mathcal N_{\lambda}\). Indeed, for each \(n\),
\[
J_{\lambda}'(u_{n})[u_{n}]=0, \qquad J_{\lambda}'(u_{n})[v]=0 \ \text{for all } v\in Y^{-}.
\]
Since \(u_{n}\to u_{\lambda}\) in \(Y\) and \(J_{\lambda}'\) is continuous, passing to the limit yields
\[
J_{\lambda}'(u_{\lambda})[u_{\lambda}]=0, \qquad J_{\lambda}'(u_{\lambda})[v]=0 \ \text{for all } v\in Y^{-},
\]
hence \(u_{\lambda}\in\mathcal N_{\lambda}\) by \eqref{eq3.2}. Therefore \(d_{\lambda}\) is achieved by \(u_{\lambda}\in\mathcal N_{\lambda}\).

Finally, let \(u\in\mathcal N_{\lambda}\) satisfy \(J_{\lambda}(u)=d_{\lambda}\). Then \(w=\check m(u)=u^{+}/\|u^{+}\|\in S^{+}\) and
\[
\Psi_{\lambda}(w)=J_{\lambda}(u)=\inf_{S^{+}}\Psi_{\lambda}.
\]
Thus \(w\) is a critical point of \(\Psi_{\lambda}\), and Corollary~\ref{Cor3.8}(c) implies that \(\hat m(w)\in\mathcal N_{\lambda}\) is a critical point of \(J_{\lambda}\) with the same critical value. Since \(\hat m\big(\check m(u)\big)=u\) on \(\mathcal N_{\lambda}\), we obtain \(\hat m(w)=u\), hence \(u\) is a critical point of \(J_{\lambda}\). Therefore every minimizer on \(\mathcal N_{\lambda}\) is a nontrivial critical point of \(J_{\lambda}\), and thus a weak solution of \eqref{eq3.1}.
\end{proof}

The same compactness theorem gives the nonautonomous ground state.

\begin{Proposition}\label{Pro3.23}
Assume $(V_1)$, $(f_1)$--$(f_2)$ and $d_{V_\infty}<d_{V_\infty}^{\infty}$. Then the level $\ell$ is attained by a nonzero critical point of $I$, and
\[
d_{V_0}\le\ell\le d_{V_\infty}<d_{V_\infty}^{\infty}.
\]
\end{Proposition}

\begin{proof}
The lower bound is Lemma~\ref{Lem3.9}. Since $I\le J_{V_\infty}$ pointwise, taking maxima on each common fibre and then infima gives $\ell\le d_{V_\infty}$. Ekeland's variational principle on $S^+$ gives a minimizing Palais--Smale sequence $(w_n)$ for $\Upsilon$ at level $\ell$. By Corollary~\ref{Cor2.17}, $u_n=m(w_n)$ is a Palais--Smale sequence for $I$. Lemma~\ref{Lem3.15} excludes vanishing. Since $0<\ell<d_{V_\infty}^{\infty}$, Lemma~\ref{Lem3.20} gives convergence to a nonzero critical point with energy $\ell$.
\end{proof}

\subsection{Dependence on the constant potential}\label{Subsec3.3}

In what follows, \(I_*\Subset(-mc^2,mc^2)\) denotes an interval on which the strict gap
\[
d_\mu<d_\mu^\infty=d_{\mu,\mathbb R}
\]
holds for every \(\mu\in I_*\). By Theorem~\ref{Thm3.22}, the ground level \(d_\mu\) is then achieved for every \(\mu\in I_*\).

\begin{Proposition}\label{Pro3.24}
The function \(\lambda\mapsto d_\lambda\) is nondecreasing on \((-mc^2,mc^2)\). Moreover, it is strictly increasing on every interval \(I_*\Subset(-mc^2,mc^2)\) on which the strict gap \(d_\mu<d_\mu^\infty\) holds for all \(\mu\in I_*\).
\end{Proposition}

\begin{proof}
If \(\lambda>\mu\), then \(J_\lambda(u)\ge J_\mu(u)\) for every \(u\in Y\). Taking fibre maxima and then the infimum over \(Y^+\setminus\{0\}\) gives \(d_\mu\le d_\lambda\).

Assume now that \(\lambda,\mu\in I_*\), \(\lambda>\mu\), and \(d_\lambda=d_\mu\). By Theorem~\ref{Thm3.22}, let \(u_\lambda\) be a ground state of \(J_\lambda\). Let \(z=t u_\lambda+v\in\hat Y(u_\lambda)\) be the \(J_\mu\)-maximizer. Since \(u_\lambda\) is the \(J_\lambda\)-maximizer on the same fibre by Lemma~\ref{Lem3.1},
\[
d_\mu\le J_\mu(z)=J_\lambda(z)+\frac{\mu-\lambda}{2}\|z\|_2^2 \le d_\lambda+\frac{\mu-\lambda}{2}\|z\|_2^2.
\]
The equality \(d_\lambda=d_\mu\) forces \(z=0\), and then \(d_\mu\le0\), contradicting Lemma~\ref{Lem3.2}(a). Hence the monotonicity is strict on \(I_*\).
\end{proof}

\begin{Lemma}\label{Lem3.25}
Let \(I_*\Subset(-mc^2,mc^2)\) be an interval on which \(d_\mu<d_\mu^\infty\) holds for every \(\mu\in I_*\). Let \((\lambda_n)\subset I_*\) satisfy
\[
\lambda_1\ge\lambda_2\ge\cdots\ge\lambda_n\ge\cdots\ge\lambda, \qquad \lambda_n\to\lambda\in I_* .
\]
Then
\[
\lim_{n\to\infty}d_{\lambda_n}=d_\lambda.
\]
\end{Lemma}

\begin{proof}
Monotonicity gives \(d_\lambda\le d_{\lambda_n}\). Let \(u_\lambda\) be a ground state of \(J_\lambda\), and let \(z_n\in\hat Y(u_\lambda)\cap\mathcal N_{\lambda_n}\) be the fibre maximizer for \(J_{\lambda_n}\). Lemma~\ref{Lem3.3}, applied uniformly to the fixed fibre and to the bounded interval \([\lambda,\lambda_1]\), gives \(\|z_n\|\le C\). Hence
\[
d_{\lambda_n} \le J_{\lambda_n}(z_n) =J_\lambda(z_n)+\frac{\lambda_n-\lambda}{2}\|z_n\|_2^2 \le d_\lambda+C(\lambda_n-\lambda),
\]
which proves \(d_{\lambda_n}\to d_\lambda\).
\end{proof}

\begin{Lemma}\label{Lem3.26}
Let \(I_*\Subset(-mc^2,mc^2)\) be an interval on which \(d_\mu<d_\mu^\infty\) holds for every \(\mu\in I_*\). Let \((\lambda_n)\subset I_*\) satisfy
\[
\lambda_1\le\lambda_2\le\cdots\le\lambda_n\le\cdots\le\lambda, \qquad \lambda_n\to\lambda\in I_* .
\]
Then
\[
\lim_{n\to\infty}d_{\lambda_n}=d_\lambda.
\]
\end{Lemma}

\begin{proof}
By Proposition~\ref{Pro3.24}, \(d_{\lambda_n}\uparrow d_*\le d_\lambda\), and \(d_*>0\). Let \(u_n\) be a ground state of \(J_{\lambda_n}\). The coercivity proof in Lemma~\ref{Lem3.5} is uniform for \(\mu\in[\lambda_1,\lambda]\Subset(-mc^2,mc^2)\), hence \((u_n)\) is bounded. Moreover,
\[
J_\lambda(u_n)=d_{\lambda_n}+o(1)\to d_*, \qquad J_\lambda'(u_n)=o(1)\quad\text{in }Y^*.
\]
Thus \((u_n)\) is a \((PS)_{d_*}\) sequence for \(J_\lambda\). Since \(\lambda\in I_*\),
\[
0<d_* \le d_\lambda<d_\lambda^\infty.
\]
Applying Proposition~\ref{Pro3.21} with \(\beta=d_*\) excludes loss at infinity. Hence, up to a subsequence, \(u_n\to u\) strongly in \(Y\), where \(u\) is a nontrivial critical point of \(J_\lambda\). Proposition~\ref{Pro3.21} gives \(d_*=d_\lambda\), and the proof is complete.
\end{proof}

\begin{Proposition}\label{Pro3.27}
Let \(I_*\Subset(-mc^2,mc^2)\) be an interval on which \(d_\mu<d_\mu^\infty\) holds for every \(\mu\in I_*\). Then the function
\[
\lambda\mapsto d_\lambda
\]
is continuous on \(I_*\).
\end{Proposition}

\begin{proof}
Fix \(\lambda\in I_*\). The right and left sequential limits are given by Lemmas~\ref{Lem3.25} and \ref{Lem3.26}. Together with the monotonicity of Proposition~\ref{Pro3.24}, these one-sided limits imply continuity at \(\lambda\). Since \(\lambda\) is arbitrary, \(d_\lambda\) is continuous on \(I_*\).
\end{proof}

\subsection{The half-line problem}\label{Subsec3.4}

Assume \((f_1)\)--\((f_2)\). We construct the profile used on each terminal interval in Section~\ref{Subsec4.1} and prove the invertibility of its real linearization. Let \(\mathcal D_+\) be the Dirac operator on \(\R^+=[0,\infty)\), with domain
\[
\operatorname{dom}\mathcal D_+ =\{u\in H^1(\R^+,\C^2):u_2(0)=0\}.
\]
The endpoint condition is the degree-one case of \eqref{eq2.3}.

\begin{Lemma}\label{Lem3.28}
For every \(\lambda\in(-mc^2,mc^2)\), there exists a solution \(U_\lambda=(a_\lambda,ib_\lambda)^T\in H^2(\R^+,\C^2)\) of
\[
\mathcal D_+U_\lambda+\lambda U_\lambda=f(|U_\lambda|)U_\lambda,
\]
where \(a_\lambda,b_\lambda\) are real, \(a_\lambda(0)>0\), \(b_\lambda(0)=0\), and \(a_\lambda,b_\lambda>0\) on \((0,\infty)\). For every \(0<\gamma<\sqrt{m^2c^4-\lambda^2}/c\),
\begin{equation}\label{eq3.6}
|U_\lambda(x)|+|U_\lambda'(x)|\le C_\gamma e^{-\gamma x} \qquad(x\ge0).
\end{equation}
Moreover,
\begin{equation}\label{eq3.7}
\int_0^\infty\left(\frac12 f(|U_\lambda|)|U_\lambda|^2-F(|U_\lambda|)\right)dx =\frac12 d_\lambda^\infty>0.
\end{equation}
\end{Lemma}

\begin{proof}
For \(u=(a,ib)^T\), with \(a,b\) real, the equation becomes
\begin{equation}\label{eq3.8}
ca'=-(mc^2-\lambda+f(r))b,\qquad cb'=(f(r)-mc^2-\lambda)a,\qquad r=(a^2+b^2)^{1/2}.
\end{equation}
Its conserved Hamiltonian is
\[
H_\lambda(a,b) =\frac{mc^2+\lambda}{2}a^2-\frac{mc^2-\lambda}{2}b^2-F(r).
\]
Put \(G(t)=2F(t)/t^2\) for \(t>0\). Strict monotonicity of \(f\) gives
\[
G(t)<f(t),\qquad G'(t)=\frac{2(f(t)-G(t))}{t}>0.
\]
Also \(G(0+)=0\), and Lemma~\ref{Lem2.8} gives \(G(t)\to\infty\) as \(t\to\infty\). Choose \(a(0)=G^{-1}(mc^2+\lambda)\) and \(b(0)=0\). This initial value lies on \(H_\lambda=0\). On this curve, write \(a=r\cos\vartheta\), \(b=r\sin\vartheta\). A calculation using \eqref{eq3.8} gives
\begin{equation}\label{eq3.9}
G(r)=\lambda+mc^2\cos(2\vartheta),\qquad c\vartheta'=f(r)-G(r)>0,\qquad cr'=-mc^2r\sin(2\vartheta).
\end{equation}
Starting from \(\vartheta(0)=0\), the angle increases towards \(\vartheta_* =\frac12\arccos(-\lambda/(mc^2))\in(0,\pi/2)\), while \(r\) decreases to zero. Indeed, the orbit remains bounded, and a limiting angle smaller than \(\vartheta_*\) would keep \(r\) and \(\vartheta'\) bounded away from zero. The origin cannot be reached in finite time, by uniqueness for the locally Lipschitz system \eqref{eq3.8}. This proves global existence and positivity on the half line.

The last identity in \eqref{eq3.9} implies
\[
\frac{r'(x)}{r(x)}\longrightarrow -\frac{\sqrt{m^2c^4-\lambda^2}}c.
\]
It gives the decay of \(U_\lambda\) in \eqref{eq3.6}, and \eqref{eq3.8} gives the decay of its derivative. Differentiating the system, whose coefficients are bounded along the solution, also gives \(U_\lambda\in H^2\). Extending \(a_\lambda\) evenly and \(b_\lambda\) oddly gives a full-line solution.

We identify its energy with \(d_\lambda^\infty\). Proposition~\ref{Pro2.14} and Corollary~\ref{Cor2.17}, applied to the full-line operator, provide a minimizing Palais--Smale sequence for the autonomous functional. Such a sequence at the positive level \(d_\lambda^\infty\) is bounded by Lemma~\ref{Lem2.11}. The one-dimensional vanishing argument used in Lemma~\ref{Lem2.10} excludes vanishing. After translations, a subsequence therefore has a nonzero weak limit. Translations commute with the full-line spectral projections. Local compactness passes the equation to the limit, and Fatou's lemma for \(\frac12f(t)t^2-F(t)\ge0\) shows that the limiting critical point attains \(d_\lambda^\infty\).

Every nonzero full-line bound state is a constant phase multiple of the reflected solution just constructed, up to translation. To see this, the current \(2\operatorname{Re}(\overline{u_1}u_2)\) is constant and vanishes at infinity. The Hamiltonian above, with \(a^2,b^2\) replaced by \(|u_1|^2,|u_2|^2\), is also conserved and has value zero. Its value at a zero of \(u_1\) would force \(u_2=0\), and uniqueness would give the zero solution. Thus \(u_1\) never vanishes. The zero-current identity allows the representation \(u=(a,ib)^Te^{i\phi}\), with \(a>0\) and \(a,b\) real. The equation gives \(\phi'=0\). On the zero-Hamiltonian curve, \eqref{eq3.9} then determines the orbit uniquely up to translation. Its two halves have the same action. Applying this classification to the full-line minimizer proves \eqref{eq3.7}.
\end{proof}

For the linearized equations, we use the real Hilbert subspace
\[
\{(a,ib)^T:a,b\in L^2(\R^+,\R)\},
\]
with operator domain obtained by intersection with \(\operatorname{dom}\mathcal D_+\). The same convention will be used on \(\mathcal G\) in Section~4. Both the Dirac operator with real potential and the nonlinearity preserve this subspace. Hence a solution constructed in it solves the original complex equation. The phase direction \(iU_\lambda\) of the full linearized problem lies outside this subspace. Write \(g(z)=f(|z|)z\). Its derivative, taken over the real scalars, is
\[
g'(z)h=f(|z|)h+\frac{f'(|z|)}{|z|}\operatorname{Re}(z\cdot\overline h)z,
\]
where the second term is zero at \(z=0\). This derivative is continuous, and \(g'(0)=0\).

\begin{Lemma}\label{Lem3.29}
On the real subspace specified above, the operator
\[
\mathcal L_\lambda=\mathcal D_++\lambda-g'(U_\lambda)
\]
is an isomorphism from its \(H^1\) domain to \(L^2\). In particular,
\begin{equation}\label{eq3.10}
\|h\|_{H^1(\R^+)}\le C_\lambda\|\mathcal L_\lambda h\|_2.
\end{equation}
\end{Lemma}

\begin{proof}
The derivative \(U_\lambda'\) solves the linearized differential equation and belongs to \(H^1(\R^+)\). If \(h=(\xi,i\eta)^T\) is in the kernel, differentiating \eqref{eq3.8} shows that \((\xi,\eta)^T\) satisfies a real two-dimensional linear system. With \(r=|U_\lambda|\), the diagonal entries of its coefficient matrix are
\[
-\frac{f'(r)a_\lambda b_\lambda}{cr}, \qquad \frac{f'(r)a_\lambda b_\lambda}{cr}.
\]
Its trace is zero, so the Wronskian
\[
\xi b_\lambda'-\eta a_\lambda'
\]
is constant. Both \(h\) and \(U_\lambda'\) tend to zero at infinity, so this constant is zero. The solution \(U_\lambda'\) is nonzero at every finite point, by uniqueness for the differential system. It follows that \(h=tU_\lambda'\) for some real constant \(t\).

At the endpoint, \eqref{eq3.8} and the choice of \(a_\lambda(0)\) give
\[
b_\lambda'(0) =\frac{[f(a_\lambda(0))-G(a_\lambda(0))]a_\lambda(0)}c>0, \qquad G(s)=\frac{2F(s)}{s^2}.
\]
The condition \(\eta(0)=0\) therefore forces \(t=0\). The kernel is trivial.

The operator \(\mathcal D_++\lambda\) is invertible because \(|\lambda|<mc^2\). Multiplication by \(g'(U_\lambda)\) is compact from \(H^1\) to \(L^2\), by compactness on bounded intervals and the convergence of this coefficient to zero at infinity. The self-adjoint operator \(\mathcal L_\lambda\) is consequently Fredholm of index zero. Its trivial kernel proves invertibility, and the differential equation gives \eqref{eq3.10}.
\end{proof}

\section{Multiplicity results}\label{Sec4}

We first work with the family and potentials in Theorem~\ref{Thm1.1}. The fixed data include \(\bar v<mc^2\) and the constants \(\lambda_1,\ldots,\lambda_k\). The half-line solutions and their linearized inverses from Lemmas~\ref{Lem3.28} and \ref{Lem3.29} yield states on prescribed nonempty subsets of terminal intervals. We then impose \((V_1)\)--\((V_3)\) to locate the singleton states in the minimum wells and obtain a ground state.

\subsection{Solutions near the terminal intervals}\label{Subsec4.1}

For each \(i\), let \(U_{\lambda_i}\) be the profile from Lemma~\ref{Lem3.28}. Fix
\[
0<\gamma<\min_{1\le i\le k}\frac{\sqrt{m^2c^4-\lambda_i^2}}c
\]
and let \(\eta\in C^\infty(\R,[0,1])\) satisfy \(\eta=1\) on \((-\infty,0]\) and \(\eta=0\) on \([1,\infty)\). For each \(i\), set
\[
w_i(x)=
\begin{cases}
\eta(x-R_i+1)U_{\lambda_i}(x),&x\in p_i,\\
0,&x\in\mathcal G\setminus p_i.
\end{cases}
\]
The spinor \(w_i\) satisfies the terminal condition at \(z_i\) and vanishes near the attachment vertex. Thus \(w_i\in\operatorname{dom}\mathcal D\), with support contained in \([0,R_i]\subset p_i\). As in Lemma~\ref{Lem3.29}, all linearizations below are taken on the real subspace of spinors \((a,ib)^T\), with the original vertex conditions.

Since \(\|V\|_\infty\le\bar v<mc^2\), for every \(h\in\operatorname{dom}\mathcal D\), the identity following Proposition~\ref{Pro2.6} gives
\begin{equation}\label{eq4.1}
\| (\mathcal D+V)h\|_2\ge(mc^2-\bar v)\|h\|_2, \qquad \|h\|_{H^1(\mathcal G)}\le C\|(\mathcal D+V)h\|_2.
\end{equation}
The constant depends only on \(m,c,\bar v\). Moreover, the embedding \(H^1(\mathcal G,\C^2)\hookrightarrow L^\infty(\mathcal G,\C^2)\) has a constant independent of the lengths \(R_i\). This follows from the one-dimensional estimate on intervals of length at least one and from the fixed edge lengths of \(\mathcal G_0\).

Since \(V=\lambda_i\) on the support of \(w_i\), its residual \(r_i=(\mathcal D+V)w_i-g(w_i)\) is
\[
r_i=-ic\sigma_1\eta'(x-R_i+1)U_{\lambda_i} +\eta(x-R_i+1)g(U_{\lambda_i})-g(\eta(x-R_i+1)U_{\lambda_i})
\]
on \(p_i\), and is zero on the other edges. It is supported in \([R_i-1,R_i]\). Lemma~\ref{Lem3.28} implies
\begin{equation}\label{eq4.2}
\|w_i\|_{H^1(\mathcal G)}\le C, \qquad \|r_i\|_2\le Ce^{-\gamma R_i}.
\end{equation}

For a nonempty \(S\subset\{1,\ldots,k\}\), set
\[
w_S=\sum_{i\in S}w_i,\qquad r_S=(\mathcal D+V)w_S-g(w_S)=\sum_{i\in S}r_i.
\]
The supports are disjoint, so \(\|r_S\|_2\le Ce^{-\gamma R_*}\). Here and below, the constants may depend on the fixed \(k\), but not on \(S\), the terminal lengths or the allowed potential.

\begin{Lemma}\label{Lem4.1}
For all sufficiently large \(R_*\), the real linearized operators
\[
\mathcal L_S=\mathcal D+V-g'(w_S),\qquad \varnothing\ne S\subset\{1,\ldots,k\},
\]
are invertible, and
\begin{equation}\label{eq4.3}
\|\mathcal L_S^{-1}z\|_{H^1(\mathcal G)}\le C\|z\|_2.
\end{equation}
The constant is independent of \(S\), the lengths \(R_i\), and the potential in the family of Theorem~\ref{Thm1.1}.
\end{Lemma}

\begin{proof}
The multiplication operators \(g'(w_S)\) are uniformly bounded. For each fixed graph, they are compact from the operator domain, endowed with its \(H^1\) norm, to \(L^2\). The operator \(\mathcal D+V\) is self-adjoint and invertible by \eqref{eq4.1}. Hence each \(\mathcal L_S\) is self-adjoint and Fredholm of index zero on the invariant real space.

Suppose that a uniform \(L^2\) lower bound fails as \(R_*\to\infty\). There are then graphs and potentials in the stated family, with \(R_*^{(n)}\to\infty\), nonempty subsets \(S_n\), and spinors \(h_n\) in their real operator domains such that
\[
\|h_n\|_2=1,\qquad \|\mathcal L_{S_n}h_n\|_2\to0.
\]
There are finitely many subsets, so a subsequence has \(S_n=S\). Equation \eqref{eq4.1} and boundedness of \(g'(w_S)\) give a uniform \(H^1\) bound. For every \(i\in S\), local compactness on fixed intervals \([0,A]\subset p_i\) and a diagonal extraction give a limit \(h^{(i)}\in H^1(\R^+,\C^2)\), with \(h^{(i)}_2(0)=0\). On each such interval, eventually \(V=\lambda_i\) and \(w_S=U_{\lambda_i}\). Passing to the limit gives
\[
\mathcal L_{\lambda_i}h^{(i)}=0\quad\text{on }\R^+.
\]
Lemma~\ref{Lem3.29} implies \(h^{(i)}=0\) for every \(i\in S\).

For fixed \(A\), we thus have \(g'(w_i)h_n\to0\) in \(L^2([0,A])\) on every selected edge. On its remaining part, \(g'(w_i)\) is uniformly small as \(A\to\infty\), by \eqref{eq3.6} and \(g'(0)=0\). The coefficient vanishes on all other edges. Summing over the fixed finite set \(S\) yields
\[
\|g'(w_S)h_n\|_2\longrightarrow0.
\]
Consequently \((\mathcal D+V)h_n\to0\), contrary to \eqref{eq4.1} and \(\|h_n\|_2=1\). This proves the uniform lower bound. Self-adjointness and Fredholmness give surjectivity. Applying \eqref{eq4.1} to \((\mathcal D+V)h=\mathcal L_Sh+g'(w_S)h\) proves \eqref{eq4.3}.
\end{proof}

\begin{Lemma}\label{Lem4.2}
For all sufficiently large \(R_*\), there is a nonzero critical point \(u_S\) of \(I\) for every nonempty \(S\subset\{1,\ldots,k\}\), such that
\begin{equation}\label{eq4.4}
\|u_S-w_S\|_{H^1(\mathcal G)}\le Ce^{-\gamma R_*}
\end{equation}
and
\begin{equation}\label{eq4.5}
\left|I(u_S)-\frac12\sum_{i\in S}d_{\lambda_i}^{\infty}\right|\le Ce^{-2\gamma R_*}.
\end{equation}
The constants are independent of \(S\), the lengths and the allowed potential.
\end{Lemma}

\begin{proof}
Seek a solution \(u=w_S+h\) in the real operator domain. Set
\[
T_S(h)=g(w_S+h)-g(w_S)-g'(w_S)h.
\]
The equation is equivalent to
\[
h=-\mathcal L_S^{-1}r_S+\mathcal L_S^{-1}T_S(h).
\]
The spinors \(w_S\) are uniformly bounded in \(L^\infty\). Continuity of \(g'\) on bounded subsets of \(\C^2\) and the uniform Sobolev estimate preceding \eqref{eq4.2} give, for \(\|h\|_{H^1},\|\widetilde h\|_{H^1}\le r\),
\[
\|T_S(h)-T_S(\widetilde h)\|_2\le\mu(Cr)\|h-\widetilde h\|_{H^1},\qquad \mu(t)\to0\quad(t\to0).
\]
The modulus is independent of \(S\) and the lengths. Choose \(r>0\) so that the inverse bound in \eqref{eq4.3} times \(\mu(Cr)\) is smaller than \(1/2\). For large \(R_*\), the right-hand side of the fixed point equation is a contraction of the ball of radius \(2C\|r_S\|_2\) into itself. Its fixed point satisfies \eqref{eq4.4}. The real subspace is invariant, so the resulting function solves both components of the original complex equation. Proposition~\ref{Pro2.9} gives \(I'(u_S)=0\). Since the supports of the \(w_i\) are disjoint and \(\|w_i\|_2\to\|U_{\lambda_i}\|_{L^2(\R^+)}>0\), the solution is nonzero for large \(R_*\).

Write \(\eta_i(x)=\eta(x-R_i+1)\). On the support of \(w_i\), the potential is \(\lambda_i\). The multiplier calculation in Lemma~\ref{Lem2.12}, applied on the half line, gives
\[
I(w_i)=\frac12\int_0^\infty\eta_i^2f(|U_{\lambda_i}|)|U_{\lambda_i}|^2\,dx-\int_0^\infty F(\eta_i|U_{\lambda_i}|)\,dx.
\]
By \eqref{eq3.6}--\eqref{eq3.7},
\[
I(w_i)=\frac12d_{\lambda_i}^{\infty}+O(e^{-2\gamma R_i}).
\]
The quadratic form is local on the operator domain and the supports are disjoint, so \(I(w_S)=\sum_{i\in S}I(w_i)\). On the uniformly bounded \(H^1\) neighborhoods,
\[
\left|I(w_S+h)-I(w_S)-\operatorname{Re}\langle r_S,h\rangle_2\right|\le C\|h\|_{H^1}^2.
\]
Apply this to \(h=u_S-w_S\), and use the residual estimate and \eqref{eq4.4}. This proves \eqref{eq4.5}.
\end{proof}

\begin{proof}[Proof of Theorem~\ref{Thm1.1}]
Choose \(\gamma=\min_i\sqrt{m^2c^4-\lambda_i^2}/(2c)\) and apply Lemmas~\ref{Lem4.1}--\ref{Lem4.2}. Their constants and length thresholds depend only on the fixed data in the theorem. Moreover, since \(w_S=0\) outside the selected terminal intervals and \(\|u_S\|_2\) is bounded below,
\[
\frac{\displaystyle\int_{\mathcal G\setminus\bigcup_{i\in S}[0,R_i]}|u_S|^2\,dx}{\displaystyle\int_{\mathcal G}|u_S|^2\,dx}\le Ce^{-2\gamma R_*}.
\]
Here each interval is understood as a subset of its corresponding terminal edge.

If \(S\ne T\), choose \(i\) in their symmetric difference. For every \(\vartheta\in\mathbb R\), disjointness of the supports gives \(\|w_S-e^{i\vartheta}w_T\|_2\ge\|w_i\|_2\). Therefore
\[
\inf_{\vartheta\in\mathbb R}\|u_S-e^{i\vartheta}u_T\|_2\ge\|w_i\|_2-Ce^{-\gamma R_*}>0
\]
for all sufficiently large \(R_*\). There are \(2^k-1\) nonempty subsets, which proves the theorem.
\end{proof}

\subsection{Minimum wells and ground states}\label{Subsec4.2}

Assume \((V_1)\)--\((V_3)\) and \((f_1)\)--\((f_2)\). Apply the preceding construction with \(\lambda_i=V_0\) and \(\bar v=\max\{|V_0|,|V_\infty|\}\). Write \(u^{(i)}=u_{\{i\}}\). We locate these singleton states in the minimum wells and use their energies to obtain a ground state.

\begin{Definition}\label{Def4.3}
Let \(\mathbf e_1,\dots,\mathbf e_k\) denote the standard basis of \(\mathbb R^k\). Since the tubular neighborhoods \(Z_i^{\rho_0}\) are pairwise disjoint, we define the well-label map
\[
\chi:\mathcal G\to\mathbb R^k
\]
by
\[
\chi(x)=
\begin{cases}
\mathbf e_i, & x\in Z_i^{\rho_0}\text{ for some }i\in\{1,\dots,k\},\\[3pt]
0, & x\in\mathcal G\setminus\displaystyle\bigcup_{i=1}^k Z_i^{\rho_0}.
\end{cases}
\]
Then \(\chi\in L^\infty(\mathcal G,\mathbb R^k)\) and
\[
\|\chi\|_{L^\infty}\le1.
\]

For \(u\in Y\setminus\{0\}\), define its barycenter by
\[
Q(u) = \frac{\displaystyle\int_{\mathcal G}\chi(x)|u(x)|^2\,dx} {\displaystyle\int_{\mathcal G}|u(x)|^2\,dx}.
\]
The reduced barycenter map is
\[
\mathfrak Q:S^+\to\mathbb R^k, \qquad \mathfrak Q(w)=Q(m(w)).
\]

Choose once and for all
\[
\varrho\in\Bigl(0,\frac14\Bigr).
\]
Since \(|\mathbf e_i-\mathbf e_j|=\sqrt2\) for \(i\ne j\), the closed balls
\[
\overline{B_\varrho(\mathbf e_i)} \qquad (i=1,\dots,k)
\]
are pairwise disjoint. For each \(i\in\{1,\dots,k\}\), define
\[
\Omega^i = \Bigl\{w\in S^+:\ |\mathfrak Q(w)-\mathbf e_i|<\frac{\varrho}{2}\Bigr\}.
\]
\end{Definition}

\begin{Lemma}\label{Lem4.4}
The map \(Q:Y\setminus\{0\}\to\mathbb R^k\) is continuous in the \(Y\)-topology. More precisely, if \(u_n\to u\) in \(Y\) with \(u\neq0\), then
\[
Q(u_n)\to Q(u).
\]
In particular, \(\mathfrak Q:S^+\to\mathbb R^k\) is continuous.
\end{Lemma}

\begin{proof}
If \(u_n\to u\) in \(Y\), then \(u_n\to u\) in \(L^2\). Since \(\chi\in L^\infty\),
\[
\int \chi(|u_n|^2-|u|^2)\to0, \qquad \int |u_n|^2\to\int |u|^2>0.
\]
Thus \(Q(u_n)\to Q(u)\). The continuity of \(\mathfrak Q=Q\circ m\) follows from the continuity of \(m|_{S^+}\).
\end{proof}

By Lemma~\ref{Lem4.4}, each \(\Omega^i\) is open in \(S^+\).

\begin{Lemma}\label{Lem4.5}
For every fixed \(\varrho\in(0,1/4)\) and all sufficiently large \(R_*\), the singleton critical points just defined satisfy
\begin{equation}\label{eq4.6}
\frac{\displaystyle\int_{\mathcal G\setminus Z_i}|u^{(i)}|^2\,dx} {\displaystyle\int_{\mathcal G}|u^{(i)}|^2\,dx} \le Ce^{-2\gamma R_*}.
\end{equation}
Moreover,
\[
\widetilde w^{\,i}=\frac{(u^{(i)})^+}{\|(u^{(i)})^+\|}\in\Omega^i, \qquad \Upsilon'(\widetilde w^{\,i})=0,
\]
and
\[
0<d_{V_0}\le\Upsilon(\widetilde w^{\,i})<d_{V_\infty}^{\infty}.
\]
\end{Lemma}

\begin{proof}
The spinor \(w_i\) vanishes outside \(Z_i\), and \eqref{eq4.4} gives a positive lower bound for \(\|u^{(i)}\|_2\), independent of \(i\) and the lengths. Thus \eqref{eq4.6} follows. The definition of \(Q\) yields
\[
|Q(u^{(i)})-\mathbf e_i| \le2\frac{\displaystyle\int_{\mathcal G\setminus Z_i^{\rho_0}}|u^{(i)}|^2\,dx} {\displaystyle\int_{\mathcal G}|u^{(i)}|^2\,dx} \le Ce^{-2\gamma R_*}<\frac\varrho2
\]
for large \(R_*\).

A nonzero critical point of \(I\) cannot belong to \(Y^-\), since the quadratic form there is strictly negative whereas \(f\ge0\). Hence \(u^{(i)}\in\mathcal M\). Proposition~\ref{Pro2.14} gives \(m(\widetilde w^{\,i})=u^{(i)}\), and Corollary~\ref{Cor2.17} gives \(\Upsilon'(\widetilde w^{\,i})=0\). The barycenter estimate places \(\widetilde w^{\,i}\) in \(\Omega^i\).

The fibre characterization in Section~\ref{Subsec3.1}, applied on \(\mathbb R\), and \(V_0\le V_\infty\) imply \(d_{V_0}^{\infty}\le d_{V_\infty}^{\infty}\). Consequently,
\[
\frac12d_{V_0}^{\infty}\le\frac12d_{V_\infty}^{\infty}<d_{V_\infty}^{\infty}.
\]
Equation \eqref{eq4.5} then gives \(\Upsilon(\widetilde w^{\,i})=I(u^{(i)})<d_{V_\infty}^{\infty}\) for all sufficiently large \(R_*\). The lower bound follows from Lemma~\ref{Lem3.9}.
\end{proof}

\begin{proof}[Proof of Theorem~\ref{Thm1.2}]
For the singleton solutions, Lemma~\ref{Lem4.5} gives the stated localization. Their constant-phase orbits are distinct by Theorem~\ref{Thm1.1}. Since every nonzero critical point of \(I\) belongs to \(\mathcal M\),
\[
0<d_{V_0}\le\ell\le\min_iI(u^{(i)})<d_{V_\infty}^{\infty}
\]
for all sufficiently large \(R_*\). The strict upper bound is uniform in the family, by \eqref{eq4.5} with \(S=\{i\}\) and the fixed positive margin
\[
d_{V_\infty}^{\infty}-\frac12d_{V_0}^{\infty}\ge\frac12d_{V_\infty}^{\infty}>0.
\]

To apply Proposition~\ref{Pro3.23}, we verify its autonomous gap condition on the same graphs. Apply Lemma~\ref{Lem4.2} with \(S=\{1\}\) to the auxiliary constant potential \(V\equiv V_\infty\), taking \(\lambda_i=V_\infty\) for all \(i\). This gives a nonzero autonomous critical point \(v\in\mathcal N_{V_\infty}\) such that
\[
d_{V_\infty}\le J_{V_\infty}(v)=\frac12d_{V_\infty}^{\infty}+O(e^{-2\gamma_\infty R_*})<d_{V_\infty}^{\infty},
\]
where \(\gamma_\infty=\sqrt{m^2c^4-V_\infty^2}/(2c)\), provided \(R_*\) exceeds another threshold depending only on the fixed data. Proposition~\ref{Pro3.23}, applied to the original potential \(V\), now gives a nonzero critical point \(u_*\) with \(I(u_*)=\ell\) and
\[
0<d_{V_0}\le I(u_*)=\ell\le d_{V_\infty}<d_{V_\infty}^{\infty}.
\]
Together with \(\ell\le\min_iI(u^{(i)})\) above, this proves all the assertions after increasing \(R_0\) to include both thresholds.
\end{proof}

\section*{Acknowledgment}

The authors thank the anonymous referee for the careful reading and helpful comments. This work was supported by the National Natural Science Foundation of China (12301145, 12561020, 12261107) and the Yunnan Fundamental Research Projects (202401AU070123, 202601AT070048). M. Ruzhansky is also supported by FWO Odysseus 1 Grant G.0H94.18N: Analysis and Partial Differential Equations and the Methusalem programme of the Ghent University Special Research Fund (Grant Number 01M01021).

\medskip
{\bf Data availability:}  Data sharing is not applicable to this article as no new data were created or analyzed in this study.

\medskip
{\bf Conflict of Interests:} The authors declare that there is no conflict of interest.

\end{document}